\documentclass[11pt]{article}
\usepackage{amssymb,amsmath,MnSymbol}
\usepackage{color}
\usepackage{amsthm}
\usepackage{graphicx}
\usepackage{color}
\usepackage{authblk}
\usepackage{url}
\usepackage[dvipsnames]{xcolor}
\usepackage{rotating,booktabs}
\usepackage{arydshln}
\usepackage[pagewise]{lineno}
\usepackage{lipsum}

\def\M{\mathcal M}
\def\Sym{\mbox{\rm Sym}}

\renewcommand{\phi}{\varphi}
\newcommand{\F}{\mathbb{F}}
\newcommand{\Span}[1]{\left\langle\, #1 \,\right\rangle}
\newcommand{\orho}{\overline{\rho}}
\newcommand{\AGL}{\mbox{\rm AGL}}
\newcommand{\la}{\lambda}
\newcommand{\veps}{\varepsilon}
\newcommand{\hex}[1]{{\texttt{#1}_{\text{x}}}}
\newcommand\deq{\mathrel{\stackrel{\makebox[0pt]{\mbox{\normalfont\tiny def}}}{=}}}

\newcommand{\ric}[1]{{\color{black}#1}}
\newcommand{\wt}[1]{{\color{white}#1}}

\makeatletter
\def\blfootnote{\gdef\@thefnmark{}\@footnotetext}
\makeatother

\DeclareMathOperator{\sym}{Sym}
\DeclareMathOperator{\Imm}{Im}
\DeclareMathOperator{\Ker}{Ker}
\DeclareMathOperator{\ddt}{DDT}
\DeclareMathOperator{\spann}{Span}
\DeclareMathOperator{\w}{w}

\newtheorem{theorem}{Theorem}[section]
\newtheorem{definition}[theorem]{Definition}
\newtheorem{lemma}[theorem]{Lemma}

\newtheorem{corollary}[theorem]{Corollary}
\newtheorem{remark}[theorem]{Remark}

\title{
Wave-Shaped Round Functions and Primitive Groups\blfootnote{\;\\\textit{Email adresses}: \texttt{ric.aragona@gmail.com} (R. Aragona), \texttt{marco.calderini@uib.no} (M. Calderini), \texttt{roberto.civino@univaq.it} (R. Civino), \texttt{maxsalacodes@gmail.com} (M. Sala), \texttt{ilaria.zappatore@lirmm.fr} (I. Zappatore)}
}

\author[1]{Riccardo Aragona
}
\author[2]{Marco Calderini
}
\author[1]{Roberto Civino
}
\author[3]{Massimiliano Sala
}
\author[4]{Ilaria Zappatore
 }
  
  \affil[1]{
\small{DISIM, University of L'Aquila 
}}

\affil[2]{
\small{Department of Informatics, University of Bergen 
}}
  
\affil[3]{
\small{Department of Mathematics, University of Trento 
}}

\affil[4]{
\small{LIRMM of Montpellier
}}

\date{}
\begin{document}
\maketitle
\vspace{-5mm}
\begin{abstract}
\noindent Round functions used as building blocks for iterated block ciphers, both in the case of Substitution-Permutation Networks (SPN) and Feistel Networks (FN), are often obtained as the composition of different layers which provide confusion and diffusion, and key additions. The bijectivity of any encryption function, crucial in order to make the decryption possible, is guaranteed by the use of invertible layers or by the Feistel structure. 
In this work a new family of ciphers, called \emph{wave ciphers}, is introduced. In wave ciphers, round functions feature \emph{wave functions}, which are vectorial Boolean functions obtained as the composition of \emph{non-invertible} layers, where the confusion layer enlarges the message which returns to its original size after the diffusion layer is applied. \ric{This is motivated by the fact that} relaxing the requirement that all the layers are invertible allows to consider more functions which are optimal with regard to non-linearity. \ric{In particular it allows to consider injective APN S-boxes}.
In order to guarantee efficient decryption we propose to use wave functions in Feistel Networks.
With regard to security, the immunity from some group-theoretical attacks is investigated. In particular, it is shown how to avoid that the group generated by the round functions acts imprimitively, which represents a serious flaw for the cipher. 
The primitivity of this group is derived as a consequence of a more general result, which allows to reduce the problem of proving that a given FN generates a primitive group to the one of proving that an SPN, directly related to the given FN, generates a primitive  group. Finally, a concrete instance of real-world size wave cipher is proposed as an example, and its resistance against differential and linear cryptanalysis is also established.
\end{abstract}



\section{Introduction}\label{sec:intro}
Most modern block ciphers belong to two families of symmetric cryptosystems, i.e. Substitution-Permutation Networks (SPN) and Feistel Networks (FN), and are obtained as composition of round functions. Each round function is a key-dependent permutation of the plaintext space, designed in such a way to provide both confusion and diffusion (see \cite{sha49}). Confusion is provided most of the times by means of a non-linear layer which applies Boolean functions, called S-boxes, whereas a linear map, called diffusion layer, provides diffusion. In order to perform decryption, invertible layers and the Feistel structure are used in SPN and FN, respectively.
In the framework of SPNs, which have been widely studied in last years, especially after the selection process for the NIST standard AES \cite{AES}, decryption is performed by applying in reverse order the inverse of each layer of the cipher. In the case of FNs, it is the Feistel structure itself that guarantees a fast decryption.

\paragraph{Motivation and design principles}
It is well-known that the non-linearity of the confusion layer is a crucial parameter for the security of the cipher. In particular, in order to prevent statistical attacks (e.g. differential \cite{biham} and linear \cite{linear} cryptanalysis), block ciphers' designers are interested in invertible S-boxes reaching the best possible differential uniformity, which is two. Functions satisfying such property are called almost-perfect non-linear (APN) \cite{nyberg} and are extensively studied. Unfortunately, APN permutations are known only when the dimension $s$ of the input space for the S-box is an odd number, except for the case of the Dillon's function ($s = 6$) \cite{dillon}, which nowadays represents the only isolated case \cite{canteaut2017}. It has been shown that no permutation with $s=4$ is APN \cite{calderivillasala,hou} and the problem is still without answers for $s\geq 8$. On the other hand, the cases when $s \in \{4, 8\}$ are the most used for implementation reasons.\\
In this paper we show how to define ciphers whose S-boxes are injective APN functions with $s$ inputs, $s$ even. We do this by considering non-invertible S-boxes, focusing on injective confusion layers which enlarge the message. Notice that a similar approach is considered in the block cipher CAST-128, where $8\times32$ are used \cite{CAST}. After the confusion layer is applied, a surjective diffusion layer reduces the message to its original size. By appending a key addition to the previous layers, we obtain a vectorial Boolean function which we call a \emph{wave function}.
Consequently a \emph{wave cipher} is a block cipher featuring wave functions in its structure. In order to guarantee an efficient decryption, we propose to use wave functions inside an FN-like framework.
The opposite scenario has been considered in DES \cite{DES} and Picaro \cite{PICARO}, where an expanding linear layer is followed by a compressing confusion layer. 

\paragraph{Algebraic security}
Algebraic attacks might also represent serious threats, as we elaborate further below. It is possible to link some algebraic properties of confusion / diffusion layers and some algebraic weaknesses of the corresponding cipher.
Firstly, in 1975 Coppersmith and Grossman \cite{copp} considered a set of functions
which can be used to define a block cipher and, by studying the permutation
group generated by those, they opened the way to a new branch of research 
focused on group-theoretical properties which can reveal weaknesses of the cipher itself.
As it has been proved in \cite{kalinski}, if such a group is too small, then the cipher is vulnerable
to birthday-paradox attacks. Recently, in \cite{Ca15} the authors proved that if such group is contained in an isomorphic image of the affine group of the message space induced by a hidden sum, then it is possible to embed a dangerous trapdoor on it.  More relevant in \cite{Pat}, Paterson built a DES-like cipher,  resistant to both linear and differential cryptanalysis, whose encryption functions generate an imprimitive group 
and showed how the knowledge of this trapdoor can be turned into an efficient attack to the cipher. 
\ric {For this reason, a branch of research in Symmetric Cryptography is focused on showing that the group generated by the encryption functions of a given cipher is primitive and not of affine type} (see \cite{PriPre, ACDVS, GOST_ric, ca18, ONAN, CDVS09, SW, We1, We3, We2}). In this sense, our purpose is to give sufficient conditions for the primitivity of the group generated by the round functions of a wave cipher. These conditions result naturally from our general investigation of the link between the primitivity of the group generated by the rounds of an SPN and that of an FN. In particular, we prove a general result which links the primitivity of the group generated by the round functions of an FN and the primitivity of the group generated by the rounds of an SPN-like cipher, whose round functions are the ones performed within each round of the FN.\\
\ric{In this paper we aim at proving that it is possible to define a new family of block ciphers, which may feature injective APN S-Boxes of even size, whose round functions generate a primitive group. 
We propose a general framework for block ciphers which produces provably secure ciphers, under some cryptographic assumptions, with respect to the imprimitivity attack. In order to prove the security of the given wave cipher with respect to other classical statistical attacks (e.g. linear and differential cryptanalysis), it is needed to analyse the single instance under consideration.}
\paragraph{Description of the paper}
The paper is organised as follows:
\begin{itemize}
\item  In Section \ref{sec:not} our notation is presented, as well as some basic definitions and results concerning the non-linearity of Boolean functions and primitive permutations group.
In particular, after having presented the main differences between SPNs and FNs, we introduce a notion of \emph{classical round function}, which allows 
to describe formally both cipher families in a unified way, provided the round key is used as a translation (i.e., the key addition is the usual XOR).
\item Section \ref{sec:wavecip} includes our definitions of wave functions and wave ciphers. We also show an example of an APN $4 \times 5$ S-box, which is suitable for building a strong wave function.
\item In Section \ref{sec:security} a group-theoretical result is shown, which, as a consequence, links the
primitivity of the action of an SPN with that of an FN (Theorem \ref{main1}). Thanks to Theorem \ref{main1},
we prove that the group generated by the round functions of a wave cipher
is primitive under some standard cryptographic assumptions on the underlying wave
functions (Theorem \ref{main2}). 
\item In Section \ref{sec:concreteex} it is designed a concrete example of 64-bit wave cipher by selecting an APN $4\times 5$ S-box and a $40\times 32$ diffusion layer, and its resistance against differential and linear cryptanalysis is proved. 
\item Section \ref{sec:concl} concludes the paper and discusses some open problems.
\end{itemize}


\section{Notation and preliminaries}\label{sec:not}
Throughout this paper we use the postfix notation for every function evaluation, i.e. if $f$ is a function and $x$ an element in the domain of $f$, we denote by $xf$ the evaluation of $f$ in $x.$ We denote by $\Imm f $ the range of $f$ and by $Y f^{-1}$ the pre-image of a set $Y$.\\

\noindent A \emph{block cipher} $\Phi$ is a family of key-dependent permutations 
\[\{E_K \mid E_K: \mathcal M \rightarrow \mathcal M, \, K \in \mathcal K \},\]
 where $\mathcal M$ is the message space, $\mathcal K$ the key space, and $|\mathcal M|\leq |\mathcal K|$. The permutation $E_K$ is called the \emph{encryption function induced by the master key} $K$. The block cipher $\Phi$ is called an {iterated block cipher} if there exists $r \in \mathbb N$ such that for each $K \in \mathcal K$ the encryption function $E_K$ is the composition of $r$ {round functions}, i.e. $E_K = \veps_{1,K}\,\veps_{2,K}\ldots\veps_{r,K}$. To provide efficiency, each round function is the composition of a public component provided by the designers, and a private component derived from the user-provided key by means of a public procedure known as \emph{key-schedule}.\\

\noindent In the theory of modern iterated block cipher, two frameworks are mainly considered: Substitution-Permutation Networks (see e.g. AES \cite{AES}, SERPENT \cite{SERPENT}, PRESENT \cite{PRESENT}) and Feistel Networks (see e.g. Camelia \cite{camelia}, GOST \cite{GOST}). Figure \ref{fig.rounds} depicts the more general framework of SPNs, FNs and their round functions; one can note that  inside the round function of  an FN, a function called F-function is applied to a half of the state. In both cases, the principles of confusion and diffusion suggested by Shannon \cite{sha49} are implemented by considering each round function / {F-function} as the composition of key-induced permutation as well as non-linear confusion layers and linear diffusion layers, which are invertible in the case of SPNs and preferably (but not necessarily) invertible in the case of FNs.  We now define a class of round functions for iterated block ciphers which is large enough to include the round functions of well-established SPNs e.g. AES, PRESENT, SERPENT, and the F-function of FNs like Camelia. Notice that, for sake of simplicity, atypical rounds are not considered in this description. \\

\begin{figure}
\centering
\includegraphics[scale = 0.102]{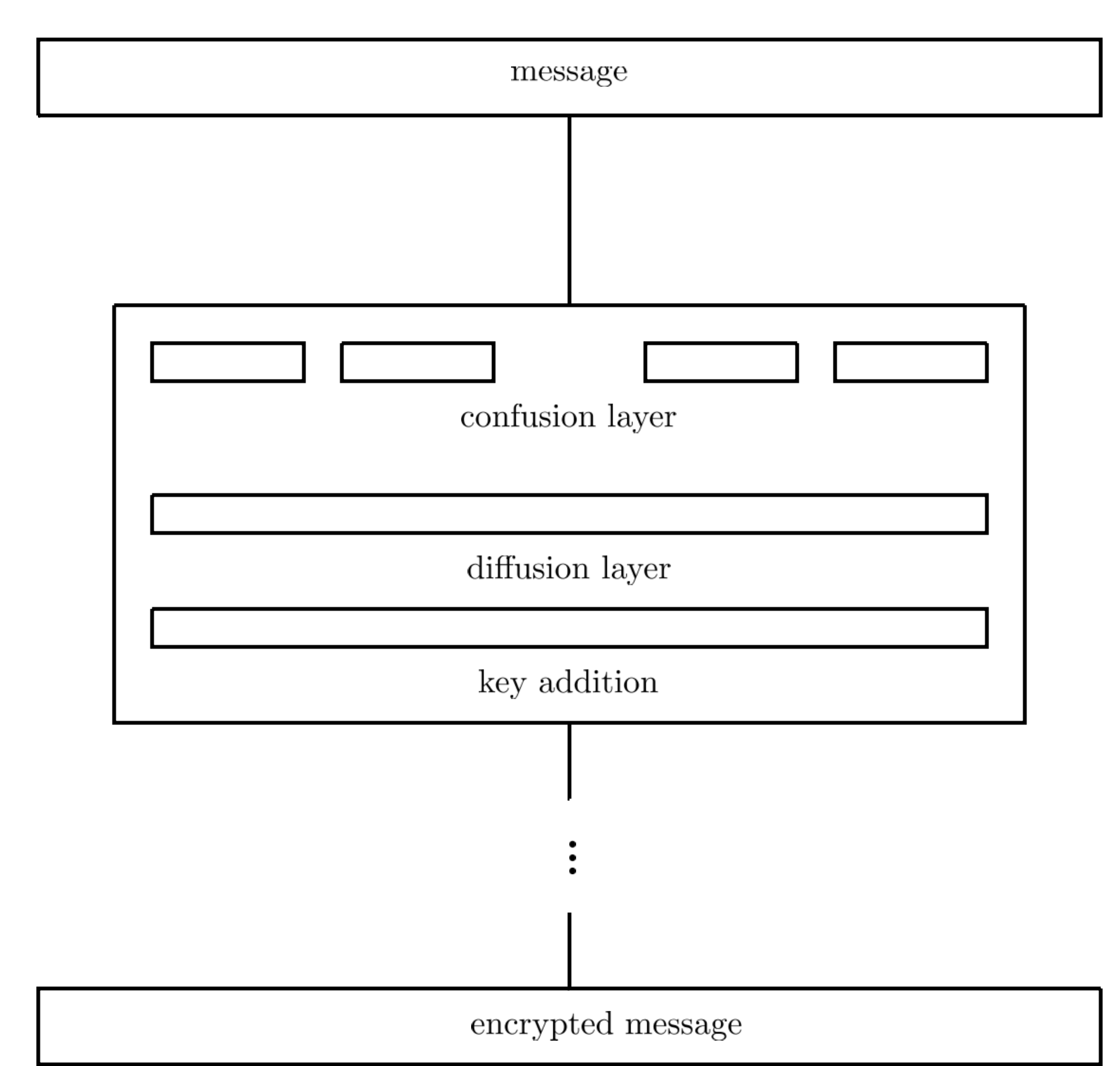}\hfil
\includegraphics[scale= 0.11]{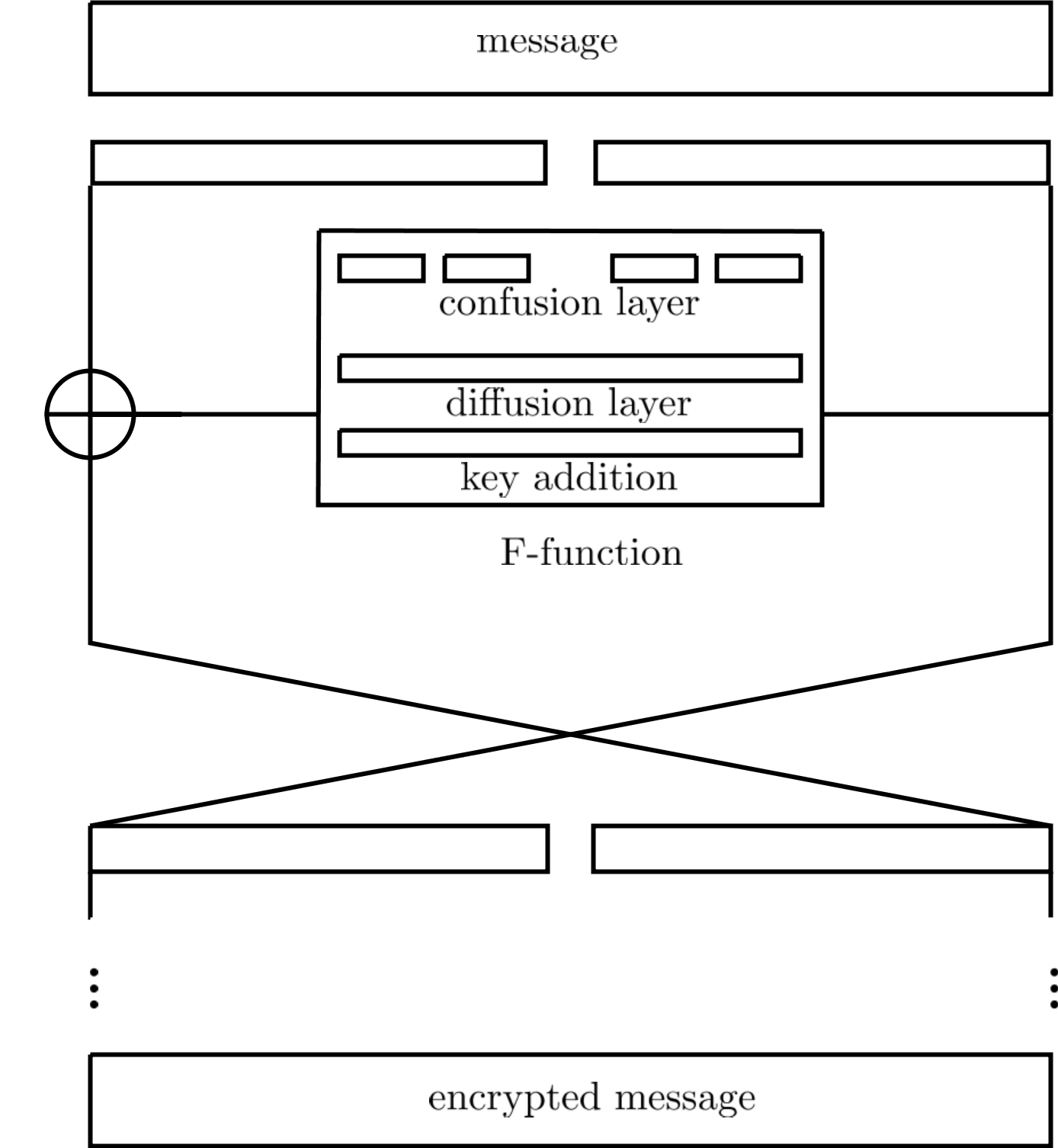}
\caption{Round function of an SPN and of an FN}\label{fig.rounds}
\end{figure}

Let $n\in \mathbb N$ and let us denote $V=(\mathbb F_2)^n$. Let us suppose $\dim(V)=n=bs$ and let us write $V= V_1\oplus V_2\oplus \ldots\oplus V_b$ where  for $1\leq j\leq b$, $\dim(V_j) = s$ and $\oplus$ represents the direct sum of vector subspaces. The subspaces $V_j$'s are called \emph{bricks}. 
We denote by $\sym(V)$ the symmetric group acting on $V$, i.e. the group of all permutations on $V$. Let us also denote by $\AGL(V)$ the group of all affine permutations of $V$, which is a primitive maximal subgroup of $\Sym(V)$.

\begin{definition}
For each $k \in V$, a \emph{classical round function} induced by $k$ is a map $\veps_{k} \in \sym(V)$ 
where $\veps_{k} = \gamma \lambda \sigma_{k}$ and 
\begin{itemize}
\item $\gamma : V \rightarrow V$ is a non-linear permutation (parallel S-box) which acts in parallel way on each $V_{j}$, i.e. \[(x_1,x_2,\ldots,x_n)\gamma = \left((x_1,\ldots,x_{s})\gamma_1,\ldots,(x_{s(b-1)+1},\ldots,x_{n})\gamma_b\right).\] The maps $\gamma_j  :  V_j \rightarrow V_j $ are traditionally called S-boxes,
 \item $\lambda \in \sym(V)$ is a linear map,
 \item  $\sigma_{k}: V \rightarrow V, x \mapsto x+k$ represents the addition with the round key $k$, where $+$ is the usual bitwise \emph{XOR}.
 \end{itemize}
\end{definition}

\noindent When used inside block ciphers, the round keys in $V$ are derived by the designer-provided key-scheduling function from the master key $K \in \mathcal K$. Since, as we will discuss later in detail, studying the role of the key-schedule is out of the scopes of this paper, one can simply suppose that round keys are stochastically  independent randomly-generated vectors in $V$. \\

In modern literature, terms ``SPN" and ``FN" may refer to a very diverse variety of ciphers. For the purposes of this paper we choose to focus only on ciphers with a XOR-based key addition. For this reason, saying SPN we refer to any cipher $\{E_K \mid K \in \mathcal K \} \subseteq \sym(\mathcal M)$ having an SPN-like structure with $\mathcal M = V$ and having classical round functions on $V$ as round functions, and saying FN to any cipher $\{E_K \mid K \in \mathcal K \} \subseteq \sym(\mathcal M)$ having an FN-like structure with  $\mathcal M = V\times V$ and having classical round functions on $V$ as F-functions. Notice that SPNs featuring a XOR-based key addition have been also called \emph{translation-based ciphers} in \cite{CDVS09}.\\

It is well-established that the security from standard statistical attacks comes from the interaction between the high non-linearity of the confusion layer and the avalanche effect guaranteed by the diffusion layer. The following section is a quick overview on one of the most used notions of non-linearity for Boolean functions, which is mainly used to prevent differential cryptanalysis \cite{biham} and other statistical attacks.


\subsection{Notions of non-linearity for Boolean functions}
Let $f:(\F_2)^s \rightarrow (\F_2)^t$ be a vectorial Boolean function and $u \in (\F_2)^s $. The derivative of $f$ in the direction $u$, denoted by $\hat{f}_u$, is the function
\[
\begin{array}{rccc}
\hat{f}_u:&(\F_2)^s&\rightarrow&(\F_2)^t\\
&x&\mapsto&xf + (x+u)f.
\end{array}
\]
\ric{The following definitions can give an estimate of the non-linearity of $f$ (see \cite{nyberg}).} 
\begin{definition}\label{deltadiff}
Let $f:(\F_2)^s \rightarrow (\F_2)^t$, $u \in (\F_2)^s$ and $v \in (\F_2)^t$. Let us define
\[
\delta_f(u,v)\deq|\{x \in (\F_2)^s \mid x\hat{f}_u=v\}| = |\,v \hat{f}_u^{\,-1}|.
\]
The \emph{difference distribution table (DDT)} of $f$  is the integer table  
\[\ddt[u, v]\deq \delta_f(u,v).\]
The \emph{differential uniformity} of $f$ is
\[
\delta(f)\deq\max_{u \ne 0}\ddt[u, v],
\]
and $f$ is said \emph{$\delta${-differentially uniform}} if $\delta=\delta(f)$. 
\end{definition}
\noindent It is well-known that $\delta(f) \geq 2$, and functions reaching the bound $\delta(f) = 2$ are called \emph{almost perfect non-linear (APN)}. Furthermore, it is easy to show that, if  $f$ is $\delta$-differentially uniform, then for each $u \in (\F_2)^s\setminus \{0\}$
\[
|\Imm(\hat{f}_u)|\geq \frac{2^s}{\delta}.
\]
The requirement of Definition \ref{deltadiff} is essentially a condition on the pre-images of the derivatives of $f$. Alternative definitions focused on the images of the derivatives of $f$ has been given e.g. in \cite{canteaut2010,CDVS09}. In particular, a function $f$ satisfying
 \[
|\Imm(\hat{f}_u)| > \frac{2^{\,s-1}}{\delta}
\]
for each $u \in (\F_2)^s \setminus \{0\}$ is called \emph{weakly $\delta$-differentially uniform} \cite{CDVS09}.  It is straightforward to verify that if $f$ is $\delta$-differentially uniform, then it is also weakly $\delta$-differentially uniform. \\


\subsection{Group generated by the round functions}
As already explained in Section \ref{sec:intro}, statistical attacks are just some of the issues that can threaten block ciphers. Several researchers have shown in recent years that also algebraic attacks can be effective. \ric{In this paper we focus on a particular group-theoretical attack, described in \cite{Pat}, based on a undesirable property of the permutation group generated by the round functions of a cipher, the \emph{imprimitivity}}. \\

Let $\Phi = \{E_K \mid K \in \mathcal K \} \subseteq \sym(\mathcal M) $ be an $r$-round iterated block cipher.
We have stressed that the group generated by all encryption functions 
\[
\Gamma(\Phi) \deq \langle E_K \mid K \in \mathcal K\rangle \leq \sym(\mathcal M)
\]
can reveal weaknesses of the cipher. However, the study of $\Gamma(\Phi)$
is not an easy task in general, since it strongly depends on the key-scheduling function (for an example of a key-schedule related study, see \cite{bea}). 
Hence one focuses on a group which is strictly related to $\Gamma(\Phi)$, which allows to ignore the effect of the key-schedule. For this reason, we do not discuss any key-schedule from now on.
Since each permutation $E_K$ is the composition of $r$ round functions $\veps_{1,K},\veps_{2,K} \ldots, \veps_{r,K}$, for each $1 \leq h \leq r$,
it is possible to define the group
\[
	\Gamma_h(\Phi)\deq \langle \veps_{h,K} \mid K \in \mathcal{K} \rangle ,
	\]
where all the possible round keys for round $h$ are considered, and so the group
\[
	\Gamma_{\infty}(\Phi)\deq\langle \Gamma_h(\Phi)\mid 1 \leq h \leq r\rangle.
\]


\subsubsection*{Imprimitive groups}
We recall some basic notions from permutation group theory. Let $G$ be a finite group acting on the set $\M$. For each $g \in G$ and $v \in \M$ we denote the action of $g$ on $v$ as $vg$. We denote by $vG=\{vg \mid g \in G\}$ the orbit of $v \in \M$ and by $G_v=\{g \in G \mid vg=v\}$ its stabiliser.
The group $G$ is said to be \emph{transitive} on $\M$ if for each $v,w \in \M$ there exists $g \in G$ such that $vg=w$.
A partition $\mathcal{B}$ of $\M$ is \emph{trivial} if $\mathcal{B}=\{\M\}$ or $\mathcal{B}=\{\{v\} \mid v \in \M\}$, and \emph{$G$-invariant} if for any $B \in \mathcal{B}$ and $g \in G$ it holds $Bg \in \mathcal{B}$. Any non-trivial and $G$-invariant partition $\mathcal{B}$ of $\M$ is called a \emph{block system}. In particular any $B \in \mathcal{B}$ is called an \emph{imprimitivity block}. The group $G$ is \emph{primitive} in its action on $\M$ (or $G$ \emph{acts primitively} on $\M$) if {$G$ is transitive and} there exists no block system. 
Otherwise,  the group $G$ is \emph{imprimitive} in its action on $\M$ (or $G$ \emph{acts imprimitively} on $\M$). We recall the following well-known results which will be useful in the remainder of the paper, and whose proofs may be found e.g. in \cite{Cam}. 

\begin{lemma}\label{lemma:block}
  A block of imprimitivity is the orbit $v  H$ of a proper subgroup  $H < G$ that properly  contains  the
  stabiliser $G_v$,  for some $v  \in \M$.
\end{lemma}

\begin{lemma}\label{lemma:trans}
  If $T$ is a transitive subgroup of $G$, then a block system
  for $G$ is also a block system for $T$.
\end{lemma}

\begin{lemma}\label{translatioBlocks}
Let us assume that $\M$ is a finite vector space over $\mathbb F_2$ and $T$ its translation group, i.e. $T = \{\sigma_v \mid  \sigma_v: \M \rightarrow \M, \:x \mapsto x+v, v \in \M\}$. 
The group $T$ is transitive and imprimitive on $\M$. Moreover, for any proper and non-trivial subgroup $U$ of $(\M , +)$,
$\{U+v \mid v \in \M\}$ is a block system.
 \end{lemma}
 
 
\subsubsection*{Imprimivity attack}

\ric{\noindent The cryptanalysts' interest into the imprimitivity of the group generated by the round functions of a block cipher arise from the study performed in \cite{Pat}, where it is showed how the imprimitivity of the group can be exploited to construct a trapdoor that may be hard to detect. In particular, the author gave an example of a DES-like cipher, which can be easily broken since its round functions generate an imprimitive group, but which is resistant to both linear and differential cryptanalysis.}


\section{Wave ciphers}\label{sec:wavecip}
The aim of this section is to define ciphers whose inner layers are not necessarily invertible, in order to use APN vectorial Boolean functions as S-boxes (even when the S-box input size is four or eight). We focus on the case of wave-shaped round functions, which feature a first layer which enlarges the state, a second which reduces its size, and a key addition. These round functions are employed in the place of classical round functions for both SPNs and FNs. To do so, let us recall that  $n=bs\in \mathbb N$ and $V=(\mathbb F_2)^n$, where $V= V_1\oplus V_2\oplus \ldots\oplus V_b$,  for $1\leq j\leq b$, and $\dim(V_j) = s$. Let us define an auxiliary space $W=(\mathbb F_2)^m$, with $n \leq m$ such that $\dim(W)=m=bt$ and $W= W_1\oplus W_2\oplus \ldots\oplus W_b$. The subspaces $W_j$'s, as the subspaces $V_j$'s, are called bricks. \\

What follows is a generalisation of the concept of classical round function.

\begin{definition}
For each $k \in V$, the \emph{wave function} induced by $k$ is a map $\veps_{k} : V \rightarrow V$,
where $\veps_{k} = \gamma \lambda \sigma_{k}$ and 
\begin{itemize}
\item $\gamma : V \rightarrow W$ is an injective non-linear transformation (parallel S-box) which acts in parallel way on each $V_{j}$, i.e. \[(x_1,x_2,\ldots,x_n)\gamma = \left((x_1,\ldots,x_{s})\gamma_1,\ldots,(x_{s(b-1)+1},\ldots,x_{n})\gamma_b\right).\] The maps $\gamma_j  :  V_j \rightarrow W_j $ are called S-boxes;
\item $\lambda : W \rightarrow V$ is a surjective linear map; 
\item $\sigma_{k}: V \rightarrow V, x \mapsto x+k$ is the round key addition. 
\end{itemize}
\end{definition}
\noindent Figure \ref{fig.wave2} depicts the composition of two consecutive wave functions. \\

\begin{figure}
\begin{center}
 \includegraphics[scale=0.14]{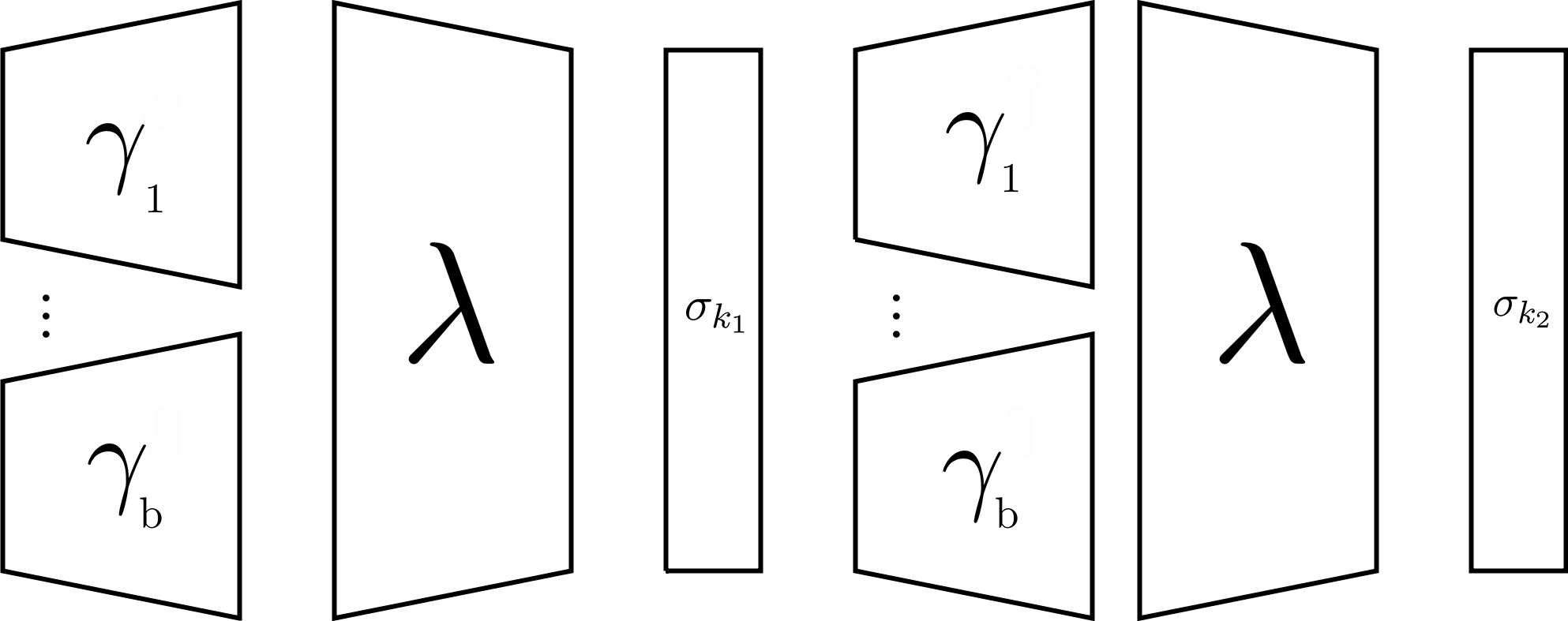}
\end{center}
\caption{Wave functions}
\label{fig.wave2}
\end{figure}

Notice that, although the hypothesis of each layer being singularly invertible may be relaxed, decryption is granted only if each wave function is overall invertible. 
The following result gives a condition on the confusion and diffusion layers which ensures that a wave function is a permutation. 

\begin{lemma}\label{perm_cond}
Let $\veps_{k} = \gamma\lambda \sigma_k$ be a wave  function. The following are equivalent: 
\begin{enumerate}
\item \label{cond_one}$\{a+b \mid a,b \in \Imm\gamma\}\cap\Ker\la = \{0\}$;
\item $\veps_k \in \sym(V)$.
\end{enumerate}
\end{lemma}
\begin{proof}
Let us assume \ref{cond_one}. Let $x_1, x_2 \in V$ such that $x_1 \veps_k = x_2 \veps_k$. Then $(x_1\gamma + x_2\gamma)\la = 0$, so 
$x_1\gamma + x_2\gamma  \in \{a+b \mid a,b \in \Imm\gamma\}\cap\Ker\la = \{0\}$, and hence $x_1\gamma =  x_2\gamma$. Since $\gamma$ is injective, it follows $x_1 = x_2$. Conversely, let $x \in \{a+b \mid a,b \in \Imm\gamma\}\cap\Ker\la$. Then there exist $x_1, x_2 \in V$ such that $x = x_1\gamma + x_2 \gamma$ and $x\la = 0$, that is $x_1 \gamma\la = x_2 \gamma\la$. Therefore $x_1 \veps_k = x_2 \veps_k$ and hence $x_1 = x_2$, which implies $x = 0$. 
\end{proof}
\begin{remark}\label{remdisj}
\emph{
Notice that it always holds $0 \in \{a+b \mid a,b \in \Imm\gamma\}\cap\Ker\la$.
Moreover, notice that if we assume that $0 \gamma = 0$, then the first condition of the previous lemma implies that $\Imm\gamma \cap \Ker\la = \{0\}.$
}
\end{remark}
\subsection{Using a 4x5 APN function}\label{sec:exa}
The function $\gamma_1: (\F_2)^4 \rightarrow (\F_2)^5$ displayed in Figure \ref{sbox} represents an example of a 4x5 injective function, which is APN, as it can be noted looking at its DDT displayed in Table \ref{tab:ddt} in the last page of this paper. Each vector is interpreted as a binary number, most significant bit first, and then represented using the hexadecimal notation (e.g. $(0,0,0,1) = \hex{1}$).  With an eye on using this function as an S-box for a wave function, one has to verify that there exists a diffusion layer satisfying the hypothesis of Lemma  \ref{perm_cond}. 
It holds $\Imm(\gamma_1) \subset (\F_2)^5$; moreover it is easy to check that $|\{a+b \mid a,b \in \Imm(\gamma_1)\}|= 31$, and the missing vector in $(\F_2)^5$ is $\xi \deq \hex{11}$. A possible way to design a cipher whose confusion layer applies in parallel $b$ copies of the S-box $\gamma_1$ is to determine a diffusion layer $\la$ whose null space is $\spann_{\F_2}\left\{(\xi,0,\ldots,0), (0,\xi,0,\ldots, 0), \ldots, (0,0,\ldots, \xi)\right\}$, where $0$ denotes the zero vector in $(\F_2)^5$. The hypothesis \ref{cond_one} of Lemma \ref{perm_cond} is satisfied, hence all the produced wave functions are bijective. Such a diffusion layer features a \emph{parallel} kernel, i.e. $$\Ker\la = \bigoplus_{j = 1}^{b} \Ker\la \cap W_j.$$ This important feature will be also exploited in the following sections.\\ 
\noindent Notice that it is not hard to find examples of such APN functions. Indeed,  it is possible to construct an APN map $\gamma:(\mathbb{F}_2)^{n}\to(\mathbb{F}_2)^{n+1}$ by considering first a function defined over $(\mathbb{F}_2)^{n}$ and then extending its image to $(\mathbb{F}_2)^{n+1}$ by adding an extra bit.
Otherwise it is possible to embed $(\mathbb{F}_2)^{n}$ into $(\mathbb{F}_2)^{n+1}$ and then consider an APN map defined over $(\mathbb{F}_2)^{n+1}$.
The map $\gamma_1$ has been  obtained using the first approach on the power function $x \mapsto x^{-1}$.\\ 
\begin{figure}
\centering
\footnotesize
\[
\begin{array}{c||c|c|c|c|c|c|c|c|c|c|c|c|c|c|c|c}
  x &  \hex{0} & \hex{1} & \hex{2} & \hex{3} & \hex{4} & \hex{5} & \hex{6} & \hex{7} & \hex{8} & \hex{9} & \hex{A} & \hex{B} & \hex{C} & \hex{D} & \hex{E} & \hex{F} \\
  \hline
  x\gamma_1 & \hex{0} & \hex{B} & \hex{1B} & \hex{8} & \hex{1D} & \hex{17} & \hex{12}  & \hex{4} & \hex{D} & \hex{14} & \hex{1} & \hex{1E} & \hex{18} & \hex{2} & \hex{E} & \hex{7}
 \end{array}
\]
\caption{A 4x5 APN S-box}
\label{sbox}
 \end{figure}
 
\setlength{\arraycolsep}{1.2pt}
\begin{sidewaystable}[htb]
\[
\begin{array}{c:cccccccccccccccccccccccccccccccc}
& \hex{\wt{0}0} & \hex{\wt{0}1} & \hex{\wt{0}2} & \hex{\wt{0}3} & \hex{\wt{0}4} & \hex{\wt{0}5} & \hex{\wt{0}6} & \hex{\wt{0}7} & \hex{\wt{0}8} & \hex{\wt{0}9} & \hex{\wt{0}A} & \hex{\wt{0}B} & \hex{\wt{0}C} & \hex{\wt{0}D} & \hex{\wt{0}E} & \hex{\wt{0}F} & \hex{10} & \hex{11} & \hex{12} & \hex{13} & \hex{14} & \hex{15} & \hex{16} & \hex{17} & \hex{18} & \hex{19} & \hex{1A} & \hex{1B} & \hex{1C} & \hex{1D} & \hex{1E} &\hex{1F} \\
\hdashline
\hex{0} & 16 & \cdot & \cdot & \cdot & \cdot & \cdot & \cdot & \cdot & \cdot & \cdot & \cdot & \cdot & \cdot & \cdot & \cdot & \cdot & \cdot & \cdot & \cdot & \cdot & \cdot & \cdot & \cdot & \cdot & \cdot & \cdot & \cdot & \cdot & \cdot & \cdot & \cdot & \cdot \\ 
\hex{1} & \cdot & \cdot & \cdot & \cdot & \cdot & \cdot & \cdot & \cdot & \cdot & 2 & 2 & 2 & \cdot & \cdot & \cdot & \cdot & \cdot & \cdot & \cdot & 2 & \cdot & \cdot & 2 & \cdot & \cdot & 2 & 2 & \cdot & \cdot & \cdot & \cdot & 2 \\ 
\hex{2} & \cdot & \cdot & \cdot & 2 & \cdot & 2 & \cdot & \cdot & \cdot & \cdot & 2 & \cdot & 2 & \cdot & \cdot & 2 & \cdot & \cdot & \cdot & 2 & \cdot & \cdot & 2 & \cdot & \cdot & \cdot & \cdot & 2 & \cdot & \cdot & \cdot & \cdot \\ 
\hex{3} & \cdot & \cdot & \cdot & \cdot & \cdot & 2 & \cdot & \cdot & 2 & \cdot & \cdot & \cdot & 2 & \cdot & \cdot & \cdot & 2 & \cdot & \cdot & 2 & \cdot & 2 & \cdot & \cdot & \cdot & 2 & \cdot & \cdot & \cdot & \cdot & \cdot & 2 \\ 
\hex{4} & \cdot & \cdot & \cdot & \cdot & \cdot & \cdot & \cdot & \cdot & \cdot & 2 & \cdot & \cdot & 2 & \cdot & \cdot & 2 & \cdot & \cdot & \cdot & \cdot & \cdot & 2 & 2 & \cdot & \cdot & 2 & \cdot & \cdot & 2 & 2 & \cdot & \cdot \\ 
\hex{5} & \cdot & \cdot & \cdot & \cdot & \cdot & \cdot & 2 & \cdot & \cdot & \cdot & \cdot & \cdot & 2 & \cdot & \cdot & 2 & 2 & \cdot & \cdot & \cdot & \cdot & \cdot & 2 & 2 & \cdot & \cdot & 2 & \cdot & \cdot & \cdot & \cdot & 2 \\ 
\hex{6} & \cdot & \cdot & \cdot & 2 & \cdot & \cdot & 2 & \cdot & \cdot & \cdot & \cdot & \cdot & \cdot & \cdot & \cdot & 2 & \cdot & \cdot & 2 & 2 & \cdot & \cdot & \cdot & \cdot & \cdot & 2 & \cdot & \cdot & 2 & \cdot & \cdot & 2 \\ 
\hex{7} & \cdot & \cdot & \cdot & 2 & 2 & \cdot & 2 & \cdot & \cdot & \cdot & 2 & \cdot & 2 & \cdot & \cdot & \cdot & \cdot & \cdot & \cdot & \cdot & \cdot & 2 & \cdot & \cdot & \cdot & 2 & 2 & \cdot & \cdot & \cdot & \cdot & \cdot \\ 
\hex{8} & \cdot & \cdot & \cdot & 2 & \cdot & 2 & \cdot & \cdot & \cdot & \cdot & \cdot & \cdot & \cdot & 2 & \cdot & \cdot & \cdot & \cdot & \cdot & \cdot & \cdot & 2 & 2 & \cdot & \cdot & \cdot & 2 & \cdot & 2 & \cdot & \cdot & 2 \\ 
\hex{9} & \cdot & \cdot & \cdot & \cdot & \cdot & 2 & 2 & \cdot & \cdot & 2 & 2 & \cdot & \cdot & \cdot & \cdot & 2 & \cdot & \cdot & \cdot & \cdot & 2 & 2 & \cdot & \cdot & \cdot & \cdot & \cdot & \cdot & \cdot & \cdot & \cdot & 2 \\ 
\hex{A} & \cdot & 2 & \cdot & \cdot & \cdot & \cdot & 2 & \cdot & \cdot & \cdot & 2 & \cdot & \cdot & \cdot & \cdot & \cdot & 2 & \cdot & \cdot & 2 & \cdot & 2 & 2 & \cdot & \cdot & \cdot & \cdot & \cdot & 2 & \cdot & \cdot & \cdot \\ 
\hex{B} & \cdot & \cdot & \cdot & \cdot & \cdot & 2 & \cdot & \cdot & \cdot & \cdot & 2 & \cdot & \cdot & \cdot & \cdot & 2 & 2 & \cdot & \cdot & \cdot & \cdot & \cdot & \cdot & \cdot & \cdot & 2 & 2 & \cdot & 2 & \cdot & 2 & \cdot \\ 
\hex{C} & \cdot & \cdot & \cdot & 2 & \cdot & \cdot & \cdot & \cdot & \cdot & 2 & \cdot & \cdot & \cdot & \cdot & \cdot & 2 & 2 & \cdot & \cdot & 2 & \cdot & 2 & \cdot & \cdot & 2 & \cdot & 2 & \cdot & \cdot & \cdot & \cdot & \cdot \\ 
\hex{D} & \cdot & \cdot & 2 & \cdot & \cdot & 2 & 2 & \cdot & \cdot & 2 & \cdot & \cdot & 2 & \cdot & \cdot & \cdot & \cdot & \cdot & \cdot & 2 & \cdot & \cdot & \cdot & \cdot & \cdot & \cdot & 2 & \cdot & 2 & \cdot & \cdot & \cdot \\ 
\hex{E} & \cdot & \cdot & \cdot & 2 & \cdot & \cdot & \cdot & \cdot & \cdot & 2 & 2 & \cdot & 2 & \cdot & 2 & \cdot & 2 & \cdot & \cdot & \cdot & \cdot & \cdot & \cdot & \cdot & \cdot & \cdot & \cdot & \cdot & 2 & \cdot & \cdot & 2 \\ 
\hex{F} & \cdot & \cdot & \cdot & 2 & \cdot & 2 & 2 & 2 & \cdot & 2 & \cdot & \cdot & \cdot & \cdot & \cdot & \cdot & 2 & \cdot & \cdot & \cdot & \cdot & \cdot & 2 & \cdot & \cdot & 2 & \cdot & \cdot & \cdot & \cdot & \cdot & \cdot 
\end{array}
\]
\vspace*{7mm}\caption{Difference distribution table of the S-box $\gamma_1$ defined in Section \ref{sec:exa}}
 \label{tab:ddt}
 \vspace{-13cm}
\end{sidewaystable}

\subsection{Feistel Networks with wave functions}
Since our goal is to use the previously defined wave functions inside a cipher, we now define a \emph{wave cipher} as an FN whose F-function is a wave function. Feistel Network's straightforward decryption encourages this choice. \\

Before defining wave ciphers, we generalise a standard security requirement for diffusion layers \cite{CDVS09} to the case of surjective maps. 
\begin{definition}\label{propMix}
A \emph{wall} of V (resp. W) is any non-trivial and proper sum of bricks of $V$ (resp. $W$). A surjective linear transformation $\lambda:W \longrightarrow V$ is a \emph{proper diffusion layer} if for any wall $W'= \bigoplus_{j \in I}W_j$ of $W$ and $V'=\bigoplus_{j \in I}V_j$ of $V$, where $I \subset \{1, \ldots, b\}$, then  
	\[
	V'\lambda^{-1} \not\subset W'+\Ker\lambda.	
	\]
\end{definition}

\noindent In other terms, if $\pi: W \longrightarrow W/\Ker\la$ is the \textit{canonical projection of} $W$ onto $W/{\Ker}(\lambda)$, $\lambda$ is proper if there exists no wall $W'= \bigoplus_{j\in I}W_j$ of $W$ and $V'=\bigoplus_{j \in I}V_j$ of $V$ such that $W'\pi\la =V'$. \\

We are now ready to define our new class of block ciphers, having \linebreak $\M = V\times V$ as message space. In what follows, $0_n$ and $1_n$ denote the zero matrix of size $n\times n$ and the identity matrix of size $n$ respectively. Moreover, for any given function $f: (\F_2)^n \rightarrow (\F_2)^n$, we denote by $\overline f$ the formal operator $\overline f : (\F_2)^{2n} \rightarrow (\F_2)^{2n}$ \[\overline{f} \,\deq \,\, \begin{pmatrix}
0_n & 1_n \\
1_n & f
\end{pmatrix},\]
such that  for any $(x_1, x_2) \in (\F_2)^n \times (\F_2)^n$ acts as $(x_1, x_2)\overline f = (x_2, x_1+x_2f)$.
The latter is called the \emph{Feistel operator induced by $f$} and, as we will discuss further, allows to give an algebraic description of FNs.
\begin{definition}\label{waveDef}
An $r$-round wave cipher $\Phi$ is a family of encryption functions $\{E_K \mid K \in \mathcal K\} \subseteq \sym(V\times V)$  such that  for each $K \in \mathcal K$ the map $E_K$ is the composition of $r$ functions. More precisely
$E_K = \overline{\veps_{1,K}}\:\overline{\veps_{2,K}}\ldots\overline{\veps_{r,K}}$,
where 
$\veps_{i,K} = \gamma\la\sigma_{k_i}$ is an $n$-bit wave  function such that
\begin{itemize}
\item $\la$ is a proper diffusion layer,
\item  the key-schedule $\mathcal{K} \rightarrow V^{r}$, $K \mapsto (k_1,k_2,\ldots,k_r)$, is surjective w.r.t. any round.
\end{itemize} 
The function $\rho \deq \gamma\la$ is called the \emph{generating function} of the cipher. 
\end{definition}
\noindent Let us notice that the ciphers previously introduced are FNs featuring a wave function as F-function. Indeed, given $(x_1,x_2) \in V\times V$ one has 
\[(x_1,x_2)\overline{\veps_{i,K}} = (x_1,x_2) \begin{pmatrix}
0_n & 1_n \\
1_n & \veps_{i,K}
\end{pmatrix} = (x_2, x_1+x_2\veps_{i,K}),\] 
where the operator $\overline{\veps_{i,K}}$ induces the Feistel structure, as shown in Figure \ref{fig.fei}. 
Moreover $\overline{\veps_{i,K}}$ is invertible with the following inverse
\[
\overline{\veps_{i,K}}^{\,-1}  = \begin{pmatrix}
\veps_{i,K} & 1_n \\
1_n & 0_n
\end{pmatrix}.
\] 
It is indeed an easy check that 
\[
(x_2, x_1+x_2\veps_{i,K}) \begin{pmatrix}
\veps_{i,K} & 1_n \\
1_n & 0_n
\end{pmatrix} = (x_1, x_2).
\]
Note that, as  for any FN, the inverse $\overline{\veps_{i,K}}^{\, -1}$ of the round function $\overline{\veps_{i,K}}$ does not involve the inverse of the wave function $\veps_{i,K}$.

\begin{remark}\label{rem:GammaTrho}
\emph{
Let $T_{(0,n)} \deq \{\sigma_{(0,k)}  \mid (x_1,x_2) \mapsto (x_1, x_2 + k)\} < \sym(V \times V)$. 
Let $\rho$ be the generating function of a wave cipher $\Phi$, and $\overline\rho$ the corresponding Feistel operator
\[\overline\rho\, = \,\begin{pmatrix} 0_n & 1_n \\ 1_n & \rho\end{pmatrix}.\] Then
$\overline{\veps_{i,K}} = \overline\rho\,\sigma_{(0,k_i)}$, and so
$
\langle\, T_{(0,n)}, \overline\rho \, \rangle
$
is the group generated by the round functions of the wave cipher $\Phi$.
}
\end{remark}

\begin{figure}[]
\begin{center}
 \includegraphics[scale=0.14]{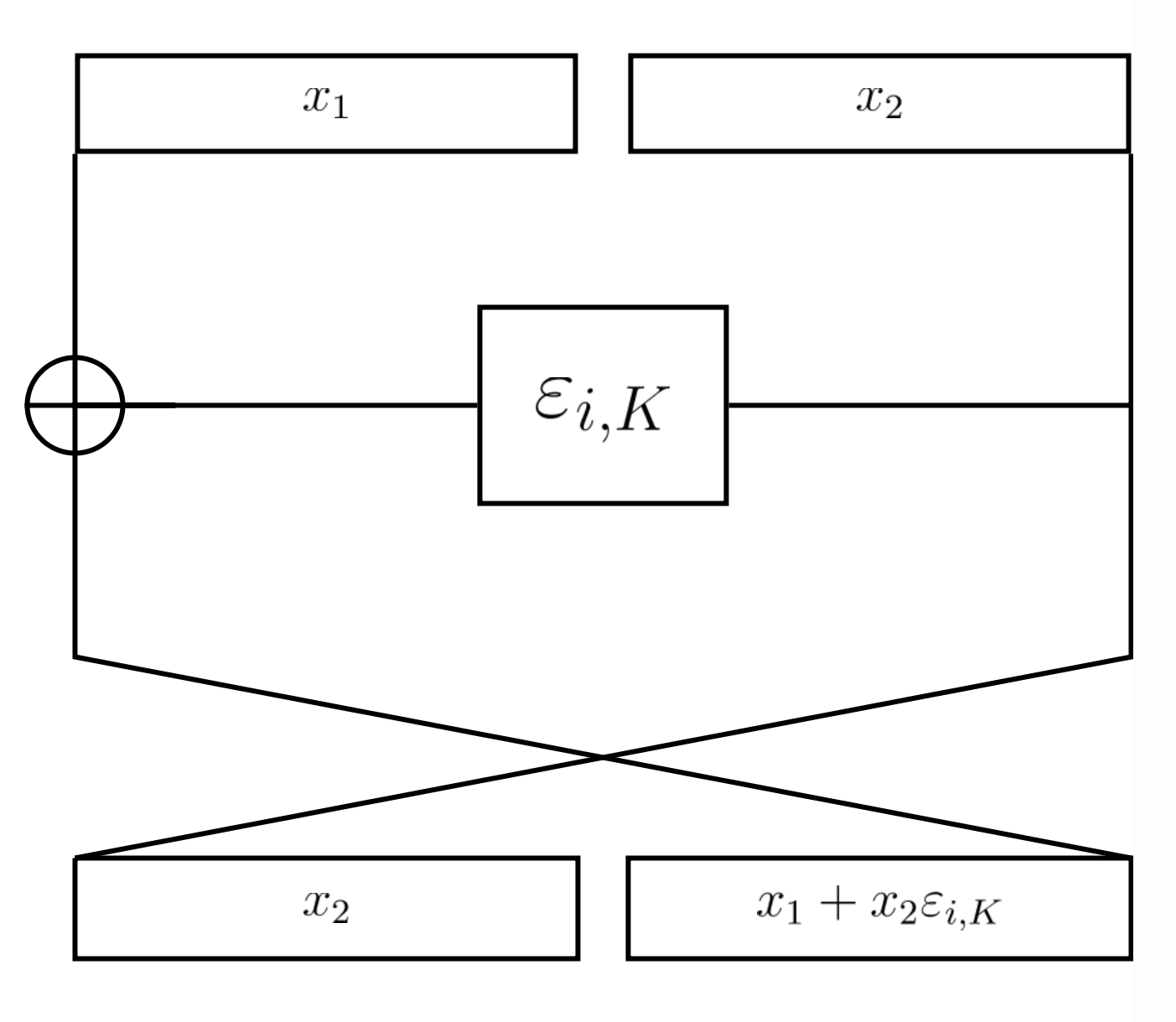}
\end{center}
\caption{Feistel structure of wave ciphers}
\label{fig.fei}
\end{figure}


\section{Group-theoretical study of Wave ciphers }\label{sec:security}
In this section, first we show a group-theoretical result which, as consequence, links the primitivity for a Substitution-Permutation Network and the primitivity for a Feistel Network having respectively round functions and F-functions with the same structure. By exploiting this result we prove that the group generated by the round functions of a wave cipher is primitive under some reasonable cryptographic assumptions on the underlying wave functions. 


\subsection{Security reduction}
Let us consider the group generated by the rounds of an FN which uses as F-functions the round functions of a primitive SPN. Here we prove a group-theoretical result which implies the primitivity of this group under the assumption that the wave functions are  invertible. In particular  this result is used to show that  the group  generated by the round functions of a wave cipher is primitive if the group\footnote{Note that the hypothesis that the wave functions are invertible allows to consider this group.} generated by the round functions of an SPN-like cipher having as round functions the same wave functions is primitive, as depicted in Fig. \ref{fig.reduc}.\\

\begin{figure}
\begin{center}
 \includegraphics[scale=0.135]{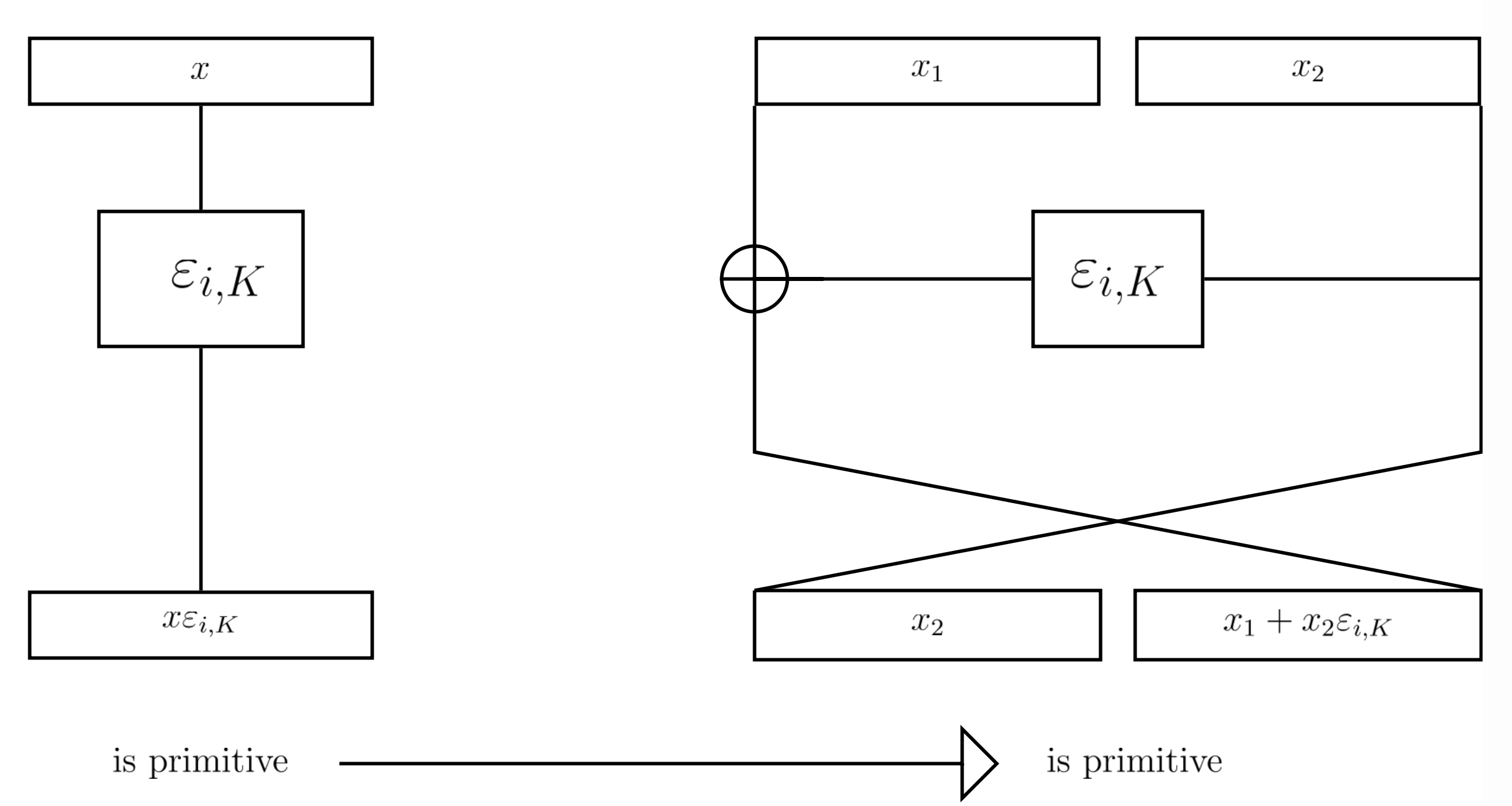}
\end{center}
\caption{Feistel to SPN reduction}
\label{fig.reduc}
\end{figure}
Let us recall that $T_{(0,n)} = \{\sigma_{(0,k)}  \mid (x_1,x_2) \mapsto (x_1, x_2 + k)\} < \sym(V \times V)$ and define
\begin{itemize}
\item $T_n \deq \{\sigma_k   \mid x \mapsto x+k\} < \sym(V)$,
\item $T_{(n,0)} \deq \{\sigma_{(k,0)}  \mid (x_1,x_2) \mapsto (x_1+k, x_2)\} < \sym(V \times V)$,
\item $T_{(n,n)} \deq \{\sigma_{(k_1,k_2)} \mid (x_1,x_2) \mapsto (x_1 + k_1, x_2 + k_2)\} < \sym(V \times V).$
\end{itemize}
Notice that $T_n \cong T_{(0,n)} \cong T_{(n,0)} < T_{(n,n)}$. \\

Let $\rho$ be any element in $\sym(V)$, $\overline\rho$ be the corresponding Feistel operator, and let $\Gamma\deq \langle\, T_{(0,n)}, \overline\rho \, \rangle.$ Since we aim at characterising imprimitivity blocks for $\Gamma$ using Lemma \ref{lemma:trans} and Lemma \ref{translatioBlocks}, we need to individuate a transitive subgroup of $\Gamma$. For this reason, the following alternative presentation of $\Gamma$ is useful. 

\begin{lemma}\label{lem:Tnn}
$\Gamma = \langle\, T_{(n,n)}, \overline\rho \, \rangle$.
\end{lemma}

\begin{proof}
Obviously $\Gamma= \langle\, T_{(0,n)}, \overline\rho \, \rangle < \langle\, T_{(n,n)}, \overline\rho \, \rangle$. On the other hand, given $x_1, x_2, k \in V$ one has 
\begin{align*}
(x_1,x_2)\overline{\rho}\sigma_{(0,k)}=   & (x_1,x_2)\begin{pmatrix} 0_n & 1_n \\ 1_n & \rho\end{pmatrix}\sigma_{(0,k)} \\ 
= & (x_2,x_1+x_2\rho+k )\\ 
= & (x_1+k,x_2)\begin{pmatrix} 0_n & 1_n \\ 1_n & \rho\end{pmatrix} \\
 = & (x_1,x_2)\sigma_{(k,0)}\overline{\rho}.
\end{align*}
Hence for each $k \in V$ it holds $\overline{\rho}\sigma_{(0,k)} = \sigma_{(k,0)}\overline{\rho}$, and consequently $\sigma_{(k,0)}\in \Gamma$. Therefore for each $k_1, k_2 \in V$,  $\sigma_{(k_1,k_2)} = \sigma_{(k_1,0)}\sigma_{(0,k_2)} \in \Gamma$.
\end{proof}

\noindent Being $T_{(n,n)}$ a transitive subgroup of $\Gamma$ and noticing that the subgroups of $T_{(n,n)}$ are of the form $ \{ \sigma_{u} : u \in U \}$, where $U$ is a subgroup of $V \times V$, we obtain the following.
\begin{lemma}\label{rem1}
  If $\Gamma$ is imprimivitive in its action on $V \times V$, then a block system is made of
  the cosets  of a subgroup of  $V \times V$, i.e. it  is 
  \begin{equation*}
    \{
      U + v
      \mid
      v \in V \times V
      \},
  \end{equation*}
  where $U$ is a non-trivial and proper subgroup of $V \times V$.
\end{lemma} 
\begin{proof}
See Lemma \ref{lemma:trans} and Lemma \ref{translatioBlocks}.
\end{proof}

According to Lemma~\ref{rem1}, in order to prove that $\Gamma$ is primitive it is sufficient to prove that no subgroup of $V \times V$ is a block.
The following theorem, due to 
 Goursat~\cite[Sections 11--12]{goursat}, characterises the
subgroups of the direct product  of two groups in terms of
suitable sections of the direct factors (see
also~\cite{Petrillo}).  We apply this result to the additive group $V \times V$.

\begin{theorem}[Goursat's Lemma \cite{goursat}]
  Let $G_1$  and $G_2$ be two  groups. There
  exists  a bijection between  
  \begin{enumerate}
  \item 
    the set  of all subgroups  of the 
    direct  product  $G_1\times   G_2$,  and  
  \item 
    the  set   of  all  triples
    $(A/B,C/D,\psi )$, where 
    \begin{itemize}
    \item $A$ is a subgroup of $G_{1}$,
    \item $C$ is a subgroup of $G_{2}$,
    \item $B$ is a normal subgroup of $A$,
    \item $D$ is a normal subgroup of $C$, and
    \item $\psi: A/B\to C/D$ is a group isomorphism.
    \end{itemize}
  \end{enumerate}

\noindent In this bijection, each subgroup of $G_1\times G_2$ can be uniquely
  written as
  \begin{equation*}
    U_{\psi}= \{
      (a,c) \in A \times C 
      :
      (a + B) \psi =c + D
      \}.
  \end{equation*}
\end{theorem}

\noindent Note that the isomorphism $\psi: A/B\to C/D$ is induced by a homomorphism $\varphi: A \to C$ such that $(a+B)\psi=a\varphi + D$ for any $a\in A$, and $B\varphi\leq D$. Such homomorphism is not unique. 
\begin{lemma}\label{lemma:psiforphi}
  In the above notation, given any homomorphism $\phi$ inducing $\psi$, we have 
 \begin{equation}\label{eq:upsi}
    U_{\psi}
    =
    \{
      (a, a \varphi + d)
      :
      a \in A, d \in D
      \}.
  \end{equation}
\end{lemma}
\begin{proof}
Note first that the  right-hand side of~\eqref{eq:upsi} is contained
  in $U_{\psi}$, since for $a \in A$ and $d \in D$ we have
$
    (a + B) \psi = a\varphi + D =  a\varphi+ d
    + D, 
$
  that is, $(a, a\varphi+ d) \in U_{\psi}$. 
  Moreover $U_{\psi}$ is contained in the right-hand side
  of~\eqref{eq:upsi}. Indeed, if $(a, c) \in U_{\psi}$ we have
$
    a \varphi + D
    =
    (a + B) \psi
    =
    c + D,
$
  so that $c = a \phi + d$ for some $d \in D$.
\end{proof}
This is our main result of this section. 
\begin{theorem}\label{main1}
Let $\rho\in \mathrm{Sym}(V)\setminus \mathrm{AGL}(V)$, $\overline\rho$ be the corresponding Feistel operator, and denote by $\Gamma = \Span{ T_n, \rho }$ and by $\overline\Gamma = \Span{ T_{(0,n)}, \overline\rho }$. If $\Gamma$ is primitive on $V$, then $\overline\Gamma$ is primitive on $V \times V$.
\end{theorem}
\noindent Before proving Theorem \ref{main1}, we show how this group-theoretical result can be helpful to us.  Let $\Phi = \{E_K \mid K \in \mathcal K\} \subseteq \sym(V\times V)$ be an $r$-round wave block cipher with a bijective generating function $\rho = \gamma\lambda$. By Remark \ref{rem:GammaTrho} one has that
$
\Gamma_{\infty}(\Phi) = \langle\, T_{(0,n)}, \overline\rho \, \rangle
$
is the group generated by the round functions of the wave cipher $\Phi$. Moreover,
$
\langle\, T_{n}, \rho \, \rangle
$
is the group generated by the wave-shaped round functions of an SPN-like cipher whose round functions are ${\veps_{i,K}}= \rho\,\sigma_{(0,k_i)}$. Therefore, from Theorem \ref{main1}, next result directly follows.
\begin{corollary}\label{cor:genprim}
Let $\Phi$ be a wave cipher, $\rho\in\sym(V)$ its generating function and $\overline\rho$ the Feistel operator induced by $\rho$. If $\Span{ T_n, \rho }$ is primitive on $V$, then $\Gamma_{\infty}(\Phi) = \langle\, T_{(0,n)}, \overline\rho \, \rangle$ is primitive on $V \times V$.
\end{corollary}

\begin{proof}[Proof of Theorem \ref{main1}]
Let us suppose that $\overline{\Gamma}=\Span{ T_{(0,n)}, \overline{\rho} }=\Span{{T_{(n,n)}}, \overline{\rho} }$ is imprimitive, so there exists a non-trivial and proper subgroup $U$ of $V \times V = (\mathbb{F}_2)^n\times (\mathbb{F}_2)^n$ such that $\{U + (v_1,v_2) \mid (v_1, v_2) \in V \times V\}$ is a block system.
In particular,
\begin{equation}\label{eq:sgrblock}
U\overline{\rho}=U+(v_1,v_2)
\end{equation}
for some $(v_1,v_2) \in V\times V$. Since $(0,0)\orho=(0,0\rho)$, we can assume $v_1=0$ and $v_2=0\rho$.
With reference to Lemma \ref{lemma:psiforphi} and its notation, we have $U=\{(a, a \varphi + d) \mid a \in A, d \in D\}$, and by \eqref{eq:sgrblock}, for any $a\in A$ and $d\in D$ there exist $x\in A$ and $y\in D$ such that
\[
(a, a\varphi+ d)\begin{pmatrix} 0_n & 1_n \\ 1_n & \rho\end{pmatrix}=(x, x\varphi+ y +0\rho),
\]
that is
\[
(a\varphi+ d, a+(a\phi+d)\rho)=(x, x\varphi+ y +0\rho).
\]
Hence, it holds $x=a\phi+d$, and considering $a=0$, we obtain $D\leq A$. Otherwise, considering $d=0$, we obtain $A\phi\leq A$.
Similarly, we have
\begin{equation}\label{eq:sgrblock2}
U\overline{\rho}^{\,-1}=U+(v'_1,v'_2)
\end{equation}
for some $(v'_1,v'_2) \in V\times V$. Since $\orho^{\,-1}=\begin{pmatrix} \rho & 1_n \\ 1_n & 0_n\end{pmatrix}$, we can consider $v'_1=0\rho$ and $v'_2=0$.
In this case, for any $a\in A$ and $d\in D$ there exist $x\in A$ and $y\in D$ such that
\[
(a\rho+a\varphi+ d, a)=(x+0\rho, x\varphi+ y ).
\]
Hence we have  $x=a\rho+a\varphi+ d +0\rho$. Substituting $x= a \phi +d $ in $x\varphi+ y$ and being $\phi$ a homomorphism, it holds $y=a+a\rho\phi+a\varphi^2+ d\phi +0\rho\phi$. Then, considering $a=0$, we obtain $y=d\phi$, and thus $D\phi\leq D$. Now, in the general case, letting $(v_1,v_2)\in V\times V$ it holds 
\begin{equation}\label{eq:block}
(U+(v_1,v_2))\overline{\rho}=U+(v'_1,v'_2)
\end{equation}
for some $(v'_1,v'_2) \in V\times V$. By definition of $\orho$, we can take $v'_1=v_2$ and $v'_2=v_1+v_2\rho$. By Lemma \ref{lemma:psiforphi} and by \eqref{eq:block}, for any $a\in A$ and $d\in D$ there exist $x\in A$ and $y\in D$ such that
\[
(a+v_1, a\varphi+ d+v_2)\begin{pmatrix} 0_n & 1_n \\ 1_n & \rho\end{pmatrix}=(x+v_2, x\varphi+ y +v_1+v_2\rho),
\]
that is,
\[
(a\varphi+ d+v_2, a+v_1+(a\phi+d+v_2)\rho)=(x+v_2, x\varphi+ y +v_1+v_2\rho),
\]
hence we have  $x=a\varphi+ d$. Substituting $x= a \phi + d$ in $x\varphi+ y +v_1+v_2\rho$,
\[
a+v_1+(a\phi+d+v_2)\rho+a\varphi^2+v_1+v_2\rho=y+ d\varphi.
\]
Then, considering $a=0$, we obtain $(d+v_2)\rho=y+ d\varphi+v_2\rho$. Since $D\varphi\leq D$, then $y+d\varphi\in D$ and so
\[
(D+v_2)\rho=D+v_2\rho.
\]
Note that we obtain the equality since $\rho$ is a permutation. If $D\ne \{0\}, (\F_2)^n$, then we proved that the imprimitivity of $\overline\Gamma$ implies the imprimitivity of $\Gamma$. To complete the proof,  it remains to consider the cases $D=(\F_2)^n$ and $D=\{0\}$.\\
\noindent${\bf\left[D=(\F_2)^n\right]}$
We proved that $D \leq A$, and from the hypotheses holds that $D\leq C$ and $\psi$ is an isomorphism between $A/B$ and $C/D$. Since $D=(\F_2)^n$, we have $D=C=A=B=(\F_2)^n$, which contradicts that $U$ is a proper subgroup of $V\times V$.\\
\noindent ${\bf \left[D=\{0\}\right]}$ First, note that in this case $B\varphi=\{0\}$. Moreover, by Lemma \ref{lemma:psiforphi}, 
\[U=\{
      (a, a \varphi)
      \mid
      a \in A
      \},
      \]
      and by \eqref{eq:block} for any $a\in A$ there exists $x\in A$ such that
\[
(a\varphi+ v_2, a+v_1+(a\phi+v_2)\rho)=(x+v_2, x\varphi +v_1+v_2\rho).
\]
Proceedings as before, it holds
\begin{equation}\label{eq:d=0}
a+a\varphi^2=(a\phi+v_2)\rho+v_2\rho.
\end{equation} 
Note that for any $a\in B\leq A$,  $a\varphi=0$ and so we obtain $a+v_2\rho=v_2\rho$ for any $a\in B$, that is, $B=\{0\}$. Therefore, if $D=\{0\}$, also $B = \{0\}$ and so $\varphi=\psi$ is an isomorphism between $A$ and $C$. Moreover, since $A\varphi$ is contained in both $A$ and $C$, then $A=C$ and $\varphi$ is an automorphism of $A$.  If $A=\{0\}$, then $A=C=D=B=\{0\}$, which contradicts that $U$ is non-trivial. If $A$ is a proper subgroup of $(\F_2)^n$, then by \eqref{eq:d=0} and since both $a+a\varphi^2$ and $a\varphi$ belong to $A$ we have
\[
(A+v_2)\rho=A+v_2\rho,
\]
and so $\Gamma$ is imprimitive. If $A=(\F_2)^n$, in equation \eqref{eq:d=0} we can consider $v_2=0$ since $a\phi+v_2$ is an element of $A=(\F_2)^n$, so we have
\[
(a\phi)\rho=a+a\varphi^2+0\rho.
\]
Since the function $x+x\varphi^2$ is linear, we proved that $\rho\in\mathrm{AGL}(V)$, which is a contradiction.
\end{proof}


\subsection{Conditions on SPN-like wave ciphers}
In the light of Theorem \ref{main1}, given a wave cipher $\Phi$ whose generating function  $\rho$ is invertible, we obtain that the group $\Gamma_\infty(\Phi)$ is primitive if we manage to prove that the group $\Span{ T_n, \rho }$ is primitive. The latter represents the group generated by the rounds of an SPN-like cipher featuring wave functions in the place of classical round functions. Although for such a cipher it may be difficult to compute the computational inverse of the encryption functions, since it has an SPN structure with non-invertible layers, we can still study its theoretical properties. In this section we underline which properties of the generating function $\rho$ guarantee that $\Span{ T_n, \rho }$ is primitive. From now on let us assume that $\rho\in\sym(V)$.\\

Let $\rho = \gamma\la$ be the generating function of a wave cipher. We can always assume that $\gamma$ maps $0$ into $0$, since it is possible to add $0\gamma$ to the round key of the previous round. Then, since $\la$ is linear, it holds $0\rho = 0$.\\

In  the following, we define a generalisation of the notion of strong anti-invariance given in \cite{CDVS09}, which is a condition in our second main theorem. 
Let us recall that, as in Section \ref{sec:wavecip}, $V= V_1\oplus V_2\oplus \ldots\oplus V_b$ and $W= W_1\oplus W_2\oplus \ldots\oplus W_b$, with $V_j = (\F_2)^s$ and  $W_j = (\F_2)^t$ for each $j \in \{1,2,\ldots,b\}$.
 \begin{definition}
Let $j \in \{1,2,\ldots, b\}$, $\gamma_j: V_j \rightarrow W_j$ be an S-box such that $0 \gamma_j = 0$, and $\lambda : W \rightarrow V$ be a surjective linear map.  Given $0 \leq \delta < s$, $\gamma_j$  is \emph{$\delta$-non-invariant with respect to $\lambda$} if for any proper subspaces $V' < V_j$ and $W'< W_j$ such that $V' \gamma_j+ \Ker\la\cap W_j = W'$, then $\dim(W')<s-\delta$.
 \end{definition}
\noindent Notice that if $0 \leq \delta< \delta' < s$ and $\gamma_j$ is $\delta'$-non-invariant w.r.t. $\lambda$, then it is also $\delta$-non-invariant w.r.t. $\lambda$.

\begin{lemma}\label{lemmaPrim}
Let $\rho = \gamma\la\in\sym(V)$ be the generating function of a wave cipher. Then $\Span{ T_n, \rho }$ is imprimitive if and only if there exists a proper and non-trivial subgroup U of V such that $(u+v)\gamma+v\gamma \in U\lambda^{-1}$, for any $u \in U$ and $v \in V$. In this case, $\{U+v \mid v \in V\}$ is a block system for $\Span{ T_n, \rho }$.\end{lemma}

\begin{proof}
Since $T_n \leq \Span{ T_n, \rho }$, if $\Span{ T_n, \rho }$ is imprimitive, then $\{U+v \mid v \in V \}$ is a block system, for some proper and non-trivial subgroup $U$ of $V$. Let $v \in V$, then $(U+v)\rho = U+v \rho = U+v\gamma \lambda$. 
Therefore for any $u \in U$ and $v \in V$ it holds 
$
(u+v)\gamma \lambda+v\gamma\lambda \in U
$
and, since $\la$ is linear,
$
(u+v)\gamma+v\gamma \in U\lambda^{-1}.
$
\end{proof}

The following is the main result of this section.
\begin{theorem}\label{main2}
Let $\rho = \gamma\la\in\sym(V)$ be the generating function of a wave cipher $\Phi$. If there exists $1\leq\delta<s$ such that for each $j \in \{1,2,\ldots,b\}$ the S-box $\gamma_j$ is
\begin{itemize}
		\item  $2^\delta$-differentially uniform,
		\item  $\delta$-non-invariant with respect to $\lambda$,
\end{itemize}
and if $\Ker\la = \bigoplus_{j=1}^{b} \Ker\la \cap W_j$,
then $\langle T_n, \rho \rangle$ is primitive (and so it is $\Gamma_{\infty}(\Phi))$.
\end{theorem}
\begin{proof}
Suppose that $\langle T_n, \rho \rangle$ is imprimitive. For the Lemma \ref{lemmaPrim}, a block system is of the form $\{U+v \mid v \in V\}$, for any proper non-trivial subgroup $U$ of $V$.
Since $U$ is an imprimitivity block and $\rho \in \langle T_n, \rho \rangle$, $U\rho=U+v$ for some $v \in V$.  Moreover, since $0\rho=0$, we obtain $U + v = U$, and consequently 
$U\rho = U\gamma \lambda =U.$
Moreover
\begin{equation}\label{equality}
	U\gamma+\Ker\lambda=U\lambda^{-1} \subseteq W,
\end{equation}
and so $U\gamma+\Ker\lambda$ is a subspace of $W$. For $1 \leq j \leq b$, let $\pi_j:V \longrightarrow V_j$ be the $j$-th projection with respect to the decomposition $V=V_1\oplus \ldots\oplus V_b$, and $I\,\deq\left\{\,j \mid j \in\{1,\ldots,b\},  U\pi_j\ne \{0\}\right\}$. Then two cases are possible: either $U\cap V_j= V_j$ for each $j \in I$, or there exists $j \in I$ such that $U \cap V_j \neq V_j.$ \\

In the first case $U = \bigoplus_{j\in I}V_j$ is a wall.
From \eqref{equality} it holds
 \begin{equation}\label{case1prim}
  (\bigoplus_{j\in I}V_j)\gamma + \Ker \lambda = (\bigoplus_{j\in I}V_j)\lambda^{-1}.	
 \end{equation}
Since $\gamma$ is a parallel transformation, we have 
\begin{equation}\label{case2prim}
(\bigoplus_{j\in I}V_j)\gamma \subset \bigoplus_{j\in I}W_j.	
\end{equation} 
Thus, from \eqref{case1prim} and \eqref{case2prim} it follows that
\begin{equation*}
(\bigoplus_{j\in I}V_j)\lambda^{-1} \subset \bigoplus_{j\in I}W_j+\Ker\lambda,	
\end{equation*} 
which is a contradiction since $\lambda$ is proper. \\

In the second case, let us assume there exists $j \in I$ such that $U \cap V_{j} \neq V_{j}$. From \eqref{equality} we have 
\begin{equation}\label{cond9}
(U\gamma+\Ker\lambda)\cap W_j=U\lambda^{-1}\cap W_j,
\end{equation}
where, since both $\gamma$ and the kernel of $\la$ are parallel,
\begin{equation}\label{cond10}
(U\gamma+\Ker\lambda)\cap W_j = U\gamma\cap W_j+\Ker\lambda\cap W_j = (U\cap V_j)\gamma_j+\Ker\lambda\cap W_j.
\end{equation}
Indeed, let $u = (u_1\gamma_1, u_2\gamma_2,\ldots, u_b\gamma_b) \in U\gamma$, $v = (v_1,v_2,\ldots,v_b) \in \Ker\lambda$, and let us assume that $w \deq u\gamma+v \in (U\gamma+\Ker\lambda)\cap W_j$, hence $w = (0,\ldots, 0, w_j, 0, \ldots, 0)$. For $l \neq j$ we obtain $u_l\gamma_l = v_l$, hence $v_l \in \Imm{\gamma_l} \cap (\Ker\la \cap W_l).$ From Remark \ref{remdisj} and since $\Ker\la$ is parallel, we have $\Imm{\gamma_l} \cap (\Ker\la \cap W_l) = \{0\}$, therefore $v_l= u_l = 0$. 
Thus, \eqref{cond9} and \eqref{cond10} imply that 
\[
 (U\cap V_j)\gamma_j+\Ker\lambda\cap W_j = U\lambda^{-1}\cap W_j,
\]
and, since $\gamma_j$ is $\delta$-non-invariant with respect to $\la$, then
\begin{equation}\label{firstbound}
\dim{( U\lambda^{-1}\cap W_j)} < s-\delta. 
\end{equation}
Furthermore, let $u \in U$ such that $ u_j \deq u\pi_j \neq 0$ and $v_j \in V_j$. Since $\Span{ T_n, \rho }$ is imprimitive, by Lemma \ref{lemmaPrim} it follows that $(u+v_j)\gamma+v_j\gamma \in U\lambda^{-1}$. Moreover $u\gamma\in U\gamma \subset U\lambda^{-1}$, and so $u\gamma+(u+v_j)\gamma+v_j\gamma \in U\lambda^{-1}$, whose components are null,  except possibly for those of the $j$-th brick, i.e. 
\begin{equation}\label{affineder}
u_j \gamma_j+(u_j+v_j)\gamma_j+v_j\gamma_j \in U\lambda^{-1} \cap W_j,
\end{equation}
which implies that $\Imm(\hat{\gamma}_{j_{u_j}}) + u_j \gamma_j\subset U\lambda^{-1} \cap W_j.$
Being $\gamma_j$ $2^\delta$-differentially uniform, it is also $2^\delta$-weakly differentially uniform, and since $u_j \ne 0$ we obtain
\begin{equation*}
2^{s-\delta-1} <| \Imm(\hat{\gamma}_{j_{u_j}})| \leq |U\lambda^{-1} \cap W_j |,
\end{equation*}
therefore $\dim(U\lambda^{-1} \cap W_j) \geq s-d$, which contradicts \eqref{firstbound}.
\end{proof}

\noindent Notice that in the proof of Theorem \ref{main2} we actually exploited that every S-box is $2^\delta$-weakly differentially uniform. Hence, we also proved the more general following result.

\begin{theorem}
Let $\rho = \gamma\la\in\sym(V)$ be the generating function of a wave cipher $\Phi$. If there exists $1\leq\delta<s$ such that for each $j \in \{1,2,\ldots,b\}$ the S-box $\gamma_j$ is
\begin{itemize}
		\item  $2^\delta$-weakly differentially uniform,
		\item  $\delta$-non-invariant with respect to $\lambda$,
\end{itemize}
and if $\Ker\la = \bigoplus_{j=1}^{b} \Ker\la \cap W_j$,
then $\langle T_n, \rho \rangle$ is primitive (and so it is $\Gamma_{\infty}(\Phi))$.
\end{theorem}
\noindent The hypothesis of each S-box being $\delta$-non-invariant w.r.t. $\la$ in Theorem \ref{main2} can be weakened by adding a reasonable requirement on the diffusion layer. However, for this result does not exist an alternative version using the weak differential uniformity.

\begin{theorem}\label{maincoro}
Let $\rho = \gamma\la\in\sym(V)$ be the generating function of a wave cipher $\Phi$. If there exists $1\leq\delta<s$ such that for each $j \in \{1,2,\ldots,b\}$ the S-box $\gamma_j$ is
\begin{itemize}
		\item  $2^\delta$-differentially uniform,
		\item  $(\delta - 1)$-non-invariant with respect to $\lambda$,
\end{itemize}
and if the diffusion layer is such that 
\begin{itemize}
\item $\Ker\la = \bigoplus_{j=1}^{b} \Ker\la \cap W_j$,
\item $\dim(\Ker\la\cap W_j) < s - \delta$ for each $j \in \{1,2,\ldots,b\}$,
\end{itemize}
then $\langle T_n, \rho \rangle$ is primitive (and so it is $\Gamma_{\infty}(\Phi))$.
\end{theorem}
\begin{proof}
The proof proceeds exactly as that of Theorem \ref{main2}. In this slightly different setting induced from a further requirement on $\la$, we can conclude that $ U\cap V_j \ne \{0\}$. Indeed, being  \[(U\cap V_j)\gamma_j+\Ker\lambda\cap W_j = U\lambda^{-1}\cap W_j,\]
 and having $\dim(U\lambda^{-1}\cap W_j) \geq s -\delta $ and $\dim(\Ker\la\cap W_j) < s - \delta$, there must be a non-zero element in $(U\cap V_j)\gamma_j$, and consequently a non-zero element $z \in U\cap V_j$. Then, reasoning as before, using Lemma \ref{lemmaPrim} one can prove that $\Imm(\hat{\gamma}_{j_{z}}) \subset U\lambda^{-1} \cap W_j$ and $|\Imm(\hat{\gamma}_{j_{z}})|\geq 2^{s-\delta}$.  Moreover, $0 \notin \Imm(\hat{\gamma}_{j_{z}})$, since $z \ne 0$ and $\gamma_j$ is injective. Hence
 \[
|U\lambda^{-1} \cap W_j | \geq 2^{s-\delta} + 1, 
 \]
and therefore $\dim(U\lambda^{-1} \cap W_j) \geq s-\delta+1$. The hypothesis of $(\delta-1)$-non-invariance of $\gamma_j$ leads to a contradiction, hence the desired holds. 
\end{proof}

\section{The security analysis of a concrete instance of wave-cipher}\label{sec:concreteex}
In the previous sections we have introduced a new framework for block ciphers, called wave ciphers, and studied its security with respect to the imprimitivity attack. In particular we primarily aimed at determining sufficient conditions on the choice of the layers which guarantee the resistance of each wave cipher satisfying such conditions against a dangerous algebraic attack. Nevertheless also statistical attacks may represent a threat for the security of these ciphers. However, as already mentioned in Sec.~\ref{sec:intro}, security against statistical attack has to be established considering a specific instance of wave cipher. For this reason, we design a concrete example of a real-world dimension wave cipher by selecting an APN S-box and a proper diffusion layer, and we analyse its resistance against differential and linear cryptanalysis.\\

The proposed instance is a 64-bit Feistel Network featuring eight $4\times5$ APN S-boxes and a $40\times32$ matrix as diffusion layer. Let us assume $n = 32$, $m = 40$, $s= 4$, $t=5$ and $b= 8$, and let us consider again the $4\times 5$ S-box $\gamma_1$ displayed in Figure \ref{sbox}. Recall that \[|\{a+b \mid a,b \in \Imm(\gamma_1)\}|= 31\] and $\xi \deq \hex{11}\notin \{a+b \mid a,b \in \Imm(\gamma_1)\}$. Since we want to design a 32-bit invertible generating function for a wave cipher whose confusion layer $\gamma$ applies $8$ copies of the S-box $\gamma_1$ and whose diffusion layer features a parallel kernel, we determine a proper diffusion layer $\la$ such that $$\Ker\la =\spann_{\F_2}\left\{(\xi,0,0,0,0,0,0,0), (0,\xi,0, 0,0,0,0,0), \ldots , (0,0,0,0,0,0,0, \xi)\right\},$$ where $0$ denotes the zero vector in $(\F_2)^5$. The matrix displayed in Figure \ref{fig:diff} is the chosen example of such a layer. Hence we build the instance of a wave cipher considering $\rho=\gamma\la$ as a bijective generating function (see Definition~\ref{waveDef}).\\

\begin{figure}
\[\la \deq \left(
\begin{smallmatrix}
1 &\cdot &\cdot &\cdot &\cdot &\cdot &\cdot &\cdot &\cdot &\cdot &\cdot &\cdot &\cdot &\cdot &\cdot &\cdot &\cdot &\cdot &\cdot &\cdot &\cdot &\cdot &\cdot &\cdot &\cdot &\cdot &\cdot &\cdot &\cdot &\cdot &\cdot &\cdot \\
\cdot &\cdot &\cdot &\cdot &\cdot &\cdot &\cdot &\cdot &\cdot &\cdot &\cdot &\cdot &\cdot &\cdot &\cdot &\cdot &1 &\cdot &\cdot &\cdot &\cdot &\cdot &\cdot &\cdot &\cdot &\cdot &\cdot &\cdot &\cdot &\cdot &\cdot &\cdot \\
\cdot &1 &\cdot &\cdot &\cdot &\cdot &\cdot &\cdot &\cdot &\cdot &\cdot &\cdot &\cdot &\cdot &\cdot &\cdot &\cdot &\cdot &\cdot &\cdot &\cdot &\cdot &\cdot &\cdot &\cdot &\cdot &\cdot &\cdot &\cdot &\cdot &\cdot &\cdot \\
\cdot &\cdot &\cdot &\cdot &\cdot &\cdot &\cdot &\cdot &\cdot &\cdot &\cdot &\cdot &\cdot &\cdot &\cdot &\cdot &\cdot &1 &\cdot &\cdot &\cdot &\cdot &\cdot &\cdot &\cdot &\cdot &\cdot &\cdot &\cdot &\cdot &\cdot &\cdot \\
1 &\cdot &\cdot &\cdot &\cdot &\cdot &\cdot &\cdot &\cdot &\cdot &\cdot &\cdot &\cdot &\cdot &\cdot &\cdot &\cdot &\cdot &\cdot &\cdot &\cdot &\cdot &\cdot &\cdot &\cdot &\cdot &\cdot &\cdot &\cdot &\cdot &\cdot &\cdot \\
\cdot &\cdot &1 &\cdot &\cdot &\cdot &\cdot &\cdot &\cdot &\cdot &\cdot &\cdot &\cdot &\cdot &\cdot &\cdot &\cdot &\cdot &\cdot &\cdot &\cdot &\cdot &\cdot &\cdot &\cdot &\cdot &\cdot &\cdot &\cdot &\cdot &\cdot &\cdot \\
\cdot &\cdot &\cdot &\cdot &\cdot &\cdot &\cdot &\cdot &\cdot &\cdot &\cdot &\cdot &\cdot &\cdot &\cdot &\cdot &\cdot &\cdot &1 &\cdot &\cdot &\cdot &\cdot &\cdot &\cdot &\cdot &\cdot &\cdot &\cdot &\cdot &\cdot &\cdot \\
\cdot &\cdot &\cdot &1 &\cdot &\cdot &\cdot &\cdot &\cdot &\cdot &\cdot &\cdot &\cdot &\cdot &\cdot &\cdot &\cdot &\cdot &\cdot &\cdot &\cdot &\cdot &\cdot &\cdot &\cdot &\cdot &\cdot &\cdot &\cdot &\cdot &\cdot &\cdot \\
\cdot &\cdot &\cdot &\cdot &\cdot &\cdot &\cdot &\cdot &\cdot &\cdot &\cdot &\cdot &\cdot &\cdot &\cdot &\cdot &\cdot &\cdot &\cdot &1 &\cdot &\cdot &\cdot &\cdot &\cdot &\cdot &\cdot &\cdot &\cdot &\cdot &\cdot &\cdot \\
\cdot &\cdot &1 &\cdot &\cdot &\cdot &\cdot &\cdot &\cdot &\cdot &\cdot &\cdot &\cdot &\cdot &\cdot &\cdot &\cdot &\cdot &\cdot &\cdot &\cdot &\cdot &\cdot &\cdot &\cdot &\cdot &\cdot &\cdot &\cdot &\cdot &\cdot &\cdot \\
\cdot &\cdot &\cdot &\cdot &1 &\cdot &\cdot &\cdot &\cdot &\cdot &\cdot &\cdot &\cdot &\cdot &\cdot &\cdot &\cdot &\cdot &\cdot &\cdot &\cdot &\cdot &\cdot &\cdot &\cdot &\cdot &\cdot &\cdot &\cdot &\cdot &\cdot &\cdot \\
\cdot &\cdot &\cdot &\cdot &\cdot &\cdot &\cdot &\cdot &\cdot &\cdot &\cdot &\cdot &\cdot &\cdot &\cdot &\cdot &\cdot &\cdot &\cdot &\cdot &1 &\cdot &\cdot &\cdot &\cdot &\cdot &\cdot &\cdot &\cdot &\cdot &\cdot &\cdot \\
\cdot &\cdot &\cdot &\cdot &\cdot &1 &\cdot &\cdot &\cdot &\cdot &\cdot &\cdot &\cdot &\cdot &\cdot &\cdot &\cdot &\cdot &\cdot &\cdot &\cdot &\cdot &\cdot &\cdot &\cdot &\cdot &\cdot &\cdot &\cdot &\cdot &\cdot &\cdot \\
\cdot &\cdot &\cdot &\cdot &\cdot &\cdot &\cdot &\cdot &\cdot &\cdot &\cdot &\cdot &\cdot &\cdot &\cdot &\cdot &\cdot &\cdot &\cdot &\cdot &\cdot &1 &\cdot &\cdot &\cdot &\cdot &\cdot &\cdot &\cdot &\cdot &\cdot &\cdot \\
\cdot &\cdot &\cdot &\cdot &1 &\cdot &\cdot &\cdot &\cdot &\cdot &\cdot &\cdot &\cdot &\cdot &\cdot &\cdot &\cdot &\cdot &\cdot &\cdot &\cdot &\cdot &\cdot &\cdot &\cdot &\cdot &\cdot &\cdot &\cdot &\cdot &\cdot &\cdot \\
\cdot &\cdot &\cdot &\cdot &\cdot &\cdot &1 &\cdot &\cdot &\cdot &\cdot &\cdot &\cdot &\cdot &\cdot &\cdot &\cdot &\cdot &\cdot &\cdot &\cdot &\cdot &\cdot &\cdot &\cdot &\cdot &\cdot &\cdot &\cdot &\cdot &\cdot &\cdot \\
\cdot &\cdot &\cdot &\cdot &\cdot &\cdot &\cdot &\cdot &\cdot &\cdot &\cdot &\cdot &\cdot &\cdot &\cdot &\cdot &\cdot &\cdot &\cdot &\cdot &\cdot &\cdot &1 &\cdot &\cdot &\cdot &\cdot &\cdot &\cdot &\cdot &\cdot &\cdot \\
\cdot &\cdot &\cdot &\cdot &\cdot &\cdot &\cdot &1 &\cdot &\cdot &\cdot &\cdot &\cdot &\cdot &\cdot &\cdot &\cdot &\cdot &\cdot &\cdot &\cdot &\cdot &\cdot &\cdot &\cdot &\cdot &\cdot &\cdot &\cdot &\cdot &\cdot &\cdot \\
\cdot &\cdot &\cdot &\cdot &\cdot &\cdot &\cdot &\cdot &\cdot &\cdot &\cdot &\cdot &\cdot &\cdot &\cdot &\cdot &\cdot &\cdot &\cdot &\cdot &\cdot &\cdot &\cdot &1 &\cdot &\cdot &\cdot &\cdot &\cdot &\cdot &\cdot &\cdot \\
\cdot &\cdot &\cdot &\cdot &\cdot &\cdot &1 &\cdot &\cdot &\cdot &\cdot &\cdot &\cdot &\cdot &\cdot &\cdot &\cdot &\cdot &\cdot &\cdot &\cdot &\cdot &\cdot &\cdot &\cdot &\cdot &\cdot &\cdot &\cdot &\cdot &\cdot &\cdot \\
\cdot &\cdot &\cdot &\cdot &\cdot &\cdot &\cdot &\cdot &1 &\cdot &\cdot &\cdot &\cdot &\cdot &\cdot &\cdot &\cdot &\cdot &\cdot &\cdot &\cdot &\cdot &\cdot &\cdot &\cdot &\cdot &\cdot &\cdot &\cdot &\cdot &\cdot &\cdot \\
\cdot &\cdot &\cdot &\cdot &\cdot &\cdot &\cdot &\cdot &\cdot &\cdot &\cdot &\cdot &\cdot &\cdot &\cdot &\cdot &\cdot &\cdot &\cdot &\cdot &\cdot &\cdot &\cdot &\cdot &1 &\cdot &\cdot &\cdot &\cdot &\cdot &\cdot &\cdot \\
\cdot &\cdot &\cdot &\cdot &\cdot &\cdot &\cdot &\cdot &\cdot &1 &\cdot &\cdot &\cdot &\cdot &\cdot &\cdot &\cdot &\cdot &\cdot &\cdot &\cdot &\cdot &\cdot &\cdot &\cdot &\cdot &\cdot &\cdot &\cdot &\cdot &\cdot &\cdot \\
\cdot &\cdot &\cdot &\cdot &\cdot &\cdot &\cdot &\cdot &\cdot &\cdot &\cdot &\cdot &\cdot &\cdot &\cdot &\cdot &\cdot &\cdot &\cdot &\cdot &\cdot &\cdot &\cdot &\cdot &\cdot &1 &\cdot &\cdot &\cdot &\cdot &\cdot &\cdot \\
\cdot &\cdot &\cdot &\cdot &\cdot &\cdot &\cdot &\cdot &1 &\cdot &\cdot &\cdot &\cdot &\cdot &\cdot &\cdot &\cdot &\cdot &\cdot &\cdot &\cdot &\cdot &\cdot &\cdot &\cdot &\cdot &\cdot &\cdot &\cdot &\cdot &\cdot &\cdot \\
\cdot &\cdot &\cdot &\cdot &\cdot &\cdot &\cdot &\cdot &\cdot &\cdot &1 &\cdot &\cdot &\cdot &\cdot &\cdot &\cdot &\cdot &\cdot &\cdot &\cdot &\cdot &\cdot &\cdot &\cdot &\cdot &\cdot &\cdot &\cdot &\cdot &\cdot &\cdot \\
\cdot &\cdot &\cdot &\cdot &\cdot &\cdot &\cdot &\cdot &\cdot &\cdot &\cdot &\cdot &\cdot &\cdot &\cdot &\cdot &\cdot &\cdot &\cdot &\cdot &\cdot &\cdot &\cdot &\cdot &\cdot &\cdot &1 &\cdot &\cdot &\cdot &\cdot &\cdot \\
\cdot &\cdot &\cdot &\cdot &\cdot &\cdot &\cdot &\cdot &\cdot &\cdot &\cdot &1 &\cdot &\cdot &\cdot &\cdot &\cdot &\cdot &\cdot &\cdot &\cdot &\cdot &\cdot &\cdot &\cdot &\cdot &\cdot &\cdot &\cdot &\cdot &\cdot &\cdot \\
\cdot &\cdot &\cdot &\cdot &\cdot &\cdot &\cdot &\cdot &\cdot &\cdot &\cdot &\cdot &\cdot &\cdot &\cdot &\cdot &\cdot &\cdot &\cdot &\cdot &\cdot &\cdot &\cdot &\cdot &\cdot &\cdot &\cdot &1 &\cdot &\cdot &\cdot &\cdot \\
\cdot &\cdot &\cdot &\cdot &\cdot &\cdot &\cdot &\cdot &\cdot &\cdot &1 &\cdot &\cdot &\cdot &\cdot &\cdot &\cdot &\cdot &\cdot &\cdot &\cdot &\cdot &\cdot &\cdot &\cdot &\cdot &\cdot &\cdot &\cdot &\cdot &\cdot &\cdot \\
\cdot &\cdot &\cdot &\cdot &\cdot &\cdot &\cdot &\cdot &\cdot &\cdot &\cdot &\cdot &1 &\cdot &\cdot &\cdot &\cdot &\cdot &\cdot &\cdot &\cdot &\cdot &\cdot &\cdot &\cdot &\cdot &\cdot &\cdot &\cdot &\cdot &\cdot &\cdot \\
\cdot &\cdot &\cdot &\cdot &\cdot &\cdot &\cdot &\cdot &\cdot &\cdot &\cdot &\cdot &\cdot &\cdot &\cdot &\cdot &\cdot &\cdot &\cdot &\cdot &\cdot &\cdot &\cdot &\cdot &\cdot &\cdot &\cdot &\cdot &1 &\cdot &\cdot &\cdot \\
\cdot &\cdot &\cdot &\cdot &\cdot &\cdot &\cdot &\cdot &\cdot &\cdot &\cdot &\cdot &\cdot &1 &\cdot &\cdot &\cdot &\cdot &\cdot &\cdot &\cdot &\cdot &\cdot &\cdot &\cdot &\cdot &\cdot &\cdot &\cdot &\cdot &\cdot &\cdot \\
\cdot &\cdot &\cdot &\cdot &\cdot &\cdot &\cdot &\cdot &\cdot &\cdot &\cdot &\cdot &\cdot &\cdot &\cdot &\cdot &\cdot &\cdot &\cdot &\cdot &\cdot &\cdot &\cdot &\cdot &\cdot &\cdot &\cdot &\cdot &\cdot &1 &\cdot &\cdot \\
\cdot &\cdot &\cdot &\cdot &\cdot &\cdot &\cdot &\cdot &\cdot &\cdot &\cdot &\cdot &1 &\cdot &\cdot &\cdot &\cdot &\cdot &\cdot &\cdot &\cdot &\cdot &\cdot &\cdot &\cdot &\cdot &\cdot &\cdot &\cdot &\cdot &\cdot &\cdot \\
\cdot &\cdot &\cdot &\cdot &\cdot &\cdot &\cdot &\cdot &\cdot &\cdot &\cdot &\cdot &\cdot &\cdot &1 &\cdot &\cdot &\cdot &\cdot &\cdot &\cdot &\cdot &\cdot &\cdot &\cdot &\cdot &\cdot &\cdot &\cdot &\cdot &\cdot &\cdot \\
\cdot &\cdot &\cdot &\cdot &\cdot &\cdot &\cdot &\cdot &\cdot &\cdot &\cdot &\cdot &\cdot &\cdot &\cdot &\cdot &\cdot &\cdot &\cdot &\cdot &\cdot &\cdot &\cdot &\cdot &\cdot &\cdot &\cdot &\cdot &\cdot &\cdot &1 &\cdot \\
\cdot &\cdot &\cdot &\cdot &\cdot &\cdot &\cdot &\cdot &\cdot &\cdot &\cdot &\cdot &\cdot &\cdot &\cdot &1 &\cdot &\cdot &\cdot &\cdot &\cdot &\cdot &\cdot &\cdot &\cdot &\cdot &\cdot &\cdot &\cdot &\cdot &\cdot &\cdot \\
\cdot &\cdot &\cdot &\cdot &\cdot &\cdot &\cdot &\cdot &\cdot &\cdot &\cdot &\cdot &\cdot &\cdot &\cdot &\cdot &\cdot &\cdot &\cdot &\cdot &\cdot &\cdot &\cdot &\cdot &\cdot &\cdot &\cdot &\cdot &\cdot &\cdot &\cdot &1 \\
\cdot &\cdot &\cdot &\cdot &\cdot &\cdot &\cdot &\cdot &\cdot &\cdot &\cdot &\cdot &\cdot &\cdot &1 &\cdot &\cdot &\cdot &\cdot &\cdot &\cdot &\cdot &\cdot &\cdot &\cdot &\cdot &\cdot &\cdot &\cdot &\cdot &\cdot &\cdot 
 \end{smallmatrix}
 \right)
 \]
 \caption{An example of $40\times32$ proper diffusion layer with parallel kernel, where each ``$\cdot$" represents $0$.}
 \label{fig:diff}
 \end{figure}

Before analysing statistical attacks, notice that the previously defined layers satisfy the hypotheses of Theorem \ref{maincoro} with $\delta = 1$, since $\gamma_1$ is $0$-non-invariant with respect to $\Ker\la$, and consequently $\rho$ is such that the group $\langle T_n, \rho \rangle$ is primitive. Then Theorem \ref{main1} implies that the group $\Gamma_{\infty}(\Phi)$ generated by the rounds of a wave cipher having $\gamma\la$ as generating function is primitive.\\

In order to discuss resistance against differential and linear cryptanalysis, let us highlight some properties of the chosen diffusion layer, which is inspired by the one of the cipher PRESENT, even though providing slower diffusion. For such cryptanalytic purposes, proceeding as in \cite{PRESENT}, we can group the eight S-boxes into two groups, as shown in Fig~\ref{fig:groups}. 
\begin{figure}
\centering
\includegraphics[scale= 0.32]{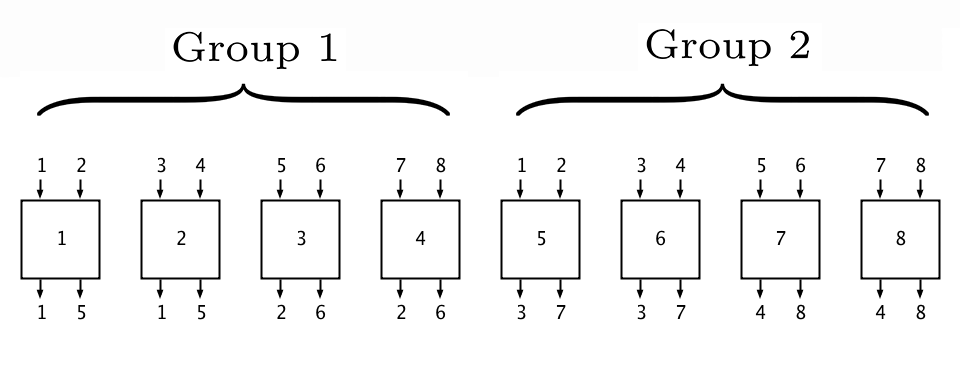}
 \caption{Diffusion properties of the matrix $\la$ of Fig.~\ref{fig:diff}.}
 \label{fig:groups}
 \end{figure}
The following properties holds: 

\begin{enumerate}
\item the input bits to an S-box come from two different S-boxes of the same group;
\item the five output bits for a particular S-box enter two different S-boxes, each of which belongs to a different group in the following round;
\item the output bits of S-boxes of different groups go to different S-boxes;
\item the branch number of $\la$ is $\min_{x \notin \ker(\la)}\left(\w_b(x)+\w_b(x\la)\right)=2$, where $\w_b(x)$ denotes the number of non-null bricks in the message $x$.
\end{enumerate}

The study of differential and linear trails, discussed in the following sections, is usually carried out assuming that the key values are random vectors of the same size as the block. For this reason, we decided not to design a concrete instance of key-scheduling algorithm for our cipher.

\subsection{Differential cryptanalysis}\label{sec:diff}
The S-box of Fig.~\ref{sbox} is APN, hence all its non-trivial differential probabilities are equal to $2^{-3}$ and any 3-round differential trail has at least 2 active S-boxes,  the worst case being the one forming the pattern 1-0-1, occurring when the XOR with the left part of the difference cancels out the output difference of the F-function for the first round. Consequently, the probability of each 3-round differential trail is upper bounded by
\[
\left({2^{-3}}\right)^2=2^{-6}.\]
Therefore, if $r=48$, the probability of a single 48-round differential trail is upper bounded by $\left(2^{-6}\right)^{16}=2^{-96}$.

\subsection{Linear cryptanalysis}\label{sec:lin}
In the case of linear cryptanalysis, the bias of all linear approximations is less or equal than $2^{-2}$. Recalling Matsui's Piling-up Lemma \cite{linear}, the maximal bias of a linear approximation of three rounds involving two active S-boxes
 is \[e_{3} =  2\times\left(2^{-2}\right)^2=2^{-3}.\]
 Consequently we can bound the maximal bias of a 48-round linear approximation by
\[
e_{48}=2^{15}\times e_{3}^{16}=2^{15}\times (2^{-3})^{16}=2^{-33}.
\]
Matsui shows in \cite{linear} that the number  of known plaintexts
required in the attack is approximatively $e^{-2}$, where $e$ denotes the maximal bias of a linear approximation. Therefore an attacker needs approximately $2^{66}$ known plaintexts to mount a key-recovery linear attack against a 48-round encryption of our instance of wave cipher.

\subsection{Other comments}
It is worth noting that, although the proposed cipher features S-boxes with an odd number of output bits, the size of the block is a power of two, which represents the optimal case for implementation needs. For example, the disadvantage of considering an FN featuring $5\times5$ APN S-boxes in place of $4\times5$ S-boxes would be twofold in terms of keeping the cipher lightweight: from one hand, the size of the block would not be a power of two; from the other hand, a $5\times5$ APN S-box requires the storage of 32 values, twice the ones needed for a $4\times5$ S-box.

\section{Conclusions and open problems}\label{sec:concl}
In this work we proposed a new family of ciphers, called \emph{wave ciphers}, whose round functions are the composition of  layers not all invertible. The round functions of a wave cipher are \emph{wave functions}, \ric{vectorial Boolean} functions obtained as the composition of injective non-linear confusion layers enlarging  the message,  surjective linear diffusion layers reducing the message size, and a key addition. Relaxing the requirement that the S-boxes are permutations allowed to consider APN functions to build confusion layers. 
In particular we gave an example of a $4\times 5$ APN S-box. We proposed to use wave functions as F-functions of Feistel Networks, where computing inverse functions is not required in order to perform decryption. 
With regard to their security we showed that, \ric{under the assumption that the generating function is invertible,} and under suitable non-linearity properties of the Boolean functions involved, the group generated by the round functions of a wave ciphers acts primitively. Finally, we presented a concrete example of 64-bit wave cipher and we proved its resistance against differential and linear cryptanalysis, as well as the imprimitivity attack.\\

\noindent Our new construction leaves several problems open, such as determining conditions on the wave functions to ensure that the group generated by the round functions of a wave cipher is the alternating group, or studying the resistance of instances of wave ciphers with respect to other more sophisticated statistical attacks on the wave-shaped structure. Moreover,  to the best of our knowledge, $s\times t$ APN functions with $s<t$ are not very much investigated in literature.
\ric{Finally note that, in order to prove that 
$\Gamma_{\infty}(\Phi) = 
\langle\, T_{(0,n)}, \overline\rho \, \rangle
$ 
is primitive, we adopted the strategy of considering an SPN having as round functions the same wave functions of $\Phi$, and we used Theorem~\ref{main1} to deduce the primitivity of  $\Gamma_{\infty}(\Phi)$ from the primitivity of  $\langle T_n, \rho \rangle$. This forced us to suppose $\rho \in \sym(V)$. However, the bijectivity of $\rho$ is not required to define a wave cipher. For this reason, one of our interests is to prove the same result in more general hypotheses on $\rho$.
}

\paragraph{Acknowledgment}
The authors are grateful to the anonymous referees for their insightful comments and suggestions, and to Andrea Visconti for several useful discussions.
This work has been partially presented at the 13th International Conference on Finite Fields and their Applications (Fq13). Some of the results showed in this paper are included in R. Civino's PhD thesis (supervised by M. Sala) and in I. Zappatore's Master thesis (supervised by R. Aragona, M. Calderini, and M. Sala). \\

\noindent R. Aragona is member of INdAM-GNSAGA (Italy). R. Civino thankfully acknowledges support by the Department of Mathematics of the University of Trento. R. Aragona, R.Civino, and M. Sala thankfully acknowledge support by MIUR-Italy via PRIN 2015TW9LSR ``Group theory and applications''. 


\end{document}